\documentclass[onefignum,onetabnum]{siamart190516}

\usepackage{diagbox}
\usepackage{times,amsmath,amsbsy,amssymb,amscd,mathrsfs}
\usepackage{cases}
\usepackage{longtable}
\usepackage{array} 
\usepackage{algorithm}
\usepackage{algorithmic}
 \usepackage{bm}
\usepackage{longtable}
\usepackage{pgfplots}
\usepackage{diagbox}
\usepackage{booktabs,makecell,multirow}
 \usepackage{url}
 \usepackage{epstopdf}

\crefname{equation}{}{}
\crefname{lem}{Lemma}{Lemmas}
\crefname{thm}{Theorem}{Theorems}
\crefname{assum}{Assumption}{Assumptions}
\crefname{rem}{Remark}{Remarks}
\crefname{line}{Line}{Lines}

\newcommand{\vect}[1]{{\rm vec}({#1} )}
\newcommand{\diag}[1]{{\rm diag}\left({#1} \right)}
\newcommand{\Vc}[0]{ {V_{\rm c}}}
\newcommand{\Vn}[0]{ {V_{\rm n}}}
\newcommand{\Vr}[0]{ {V_{\rm r}}}

\newcommand{\nm}[1]{\left\lVert {#1} \right\rVert}
\newcommand{\snm}[1]{\left\lvert {#1} \right\rvert}
\newcommand{\dual}[1]{\left\langle {#1} \right\rangle}
\newcommand{\argmin}[0]{ {\mathop{{\rm  argmin}}\,}}

\newcounter{mnote}
\let\oldmarginpar\marginpar
\renewcommand\marginpar[1]{\-\oldmarginpar[\raggedleft\footnotesize #1]%
	{\raggedright\footnotesize #1}}

\newsiamremark{assum}{Assumption} 
\newsiamthm{lem}{Lemma} 
\newsiamthm{prop}{Proposition} 
\newsiamthm{thm}{Theorem} 
\newsiamthm{coro}{Corollary} 
\newsiamthm{Def}{Definition} 
\newsiamthm{remark}{Remark}
\newsiamthm{rem}{Remark}

\newcommand{\prox}[0]{ {\bf prox}}
\newcommand{\proj}[0]{ {\rm proj}}
\newcommand{\st}[0]{ {{\rm  s.t.}\,}}
\newcommand{\R}{\,{\mathbb R}}

\newcolumntype{I}{!{\vrule width 1,5pt}}
\newlength\savedwidth

\newlength\savewidth

\begin{document}
%
%
%
%
%

	\title{An efficient semismooth Newton-AMG-based inexact primal-dual algorithm for generalized transport problems}

\author{Jun Hu\thanks{School of Mathematical Sciences, Peking University, Beijing, 100871, China (\email{hujun@math.pku.edu.cn}, \email{luohao@math.pku.edu.cn}, \email{zhang-zihang@pku.edu.cn}).}
	\and Hao Luo\footnotemark[1]
	\and Zihang Zhang\footnotemark[1]}

\headers{Semismooth Newton-AMG method for generalized transport problems}{Semismooth Newton-AMG method for generalized transport problems}

\maketitle

\begin{abstract}
	This work is concerned with the efficient optimization method for solving a large class of optimal mass transport problems. An inexact primal-dual algorithm is presented from the time discretization of a proper dynamical system, and by using the tool of Lyapunov function, the global (super-)linear convergence rate is established for function residual and feasibility violation. The proposed algorithm contains an inner problem that possesses strong semismoothness property and motivates the use of the semismooth Newton iteration. By exploring the hidden structure of the problem itself, the linear system arising from the Newton iteration is transferred equivalently into a graph Laplacian system, for which a robust algebraic multigrid method is proposed and also analyzed via the famous Xu--Zikatanov identity. Finally, numerical experiments are provided to validate the efficiency of our method.
\end{abstract}

		
	
%
%


\section{Introduction}
\label{sec:intro}
The optimal mass transport proposed by Monge, can be dated back to as early as the 1780s. Later, Kantorovich \cite{kantorovich_translocation_1942} introduced a convex relaxation of Monge's original formulation and applied it to economics. Since then, this topic attracted more attentions and it also played an increasing role in imaging processing \cite{hug_multi-physics_2015,maas_generalized_2015,2000The}, machine learning \cite{arjovsky_wasserstein_2017,courty2017optimal,kandasamy2018neural} and statistics \cite{2016Amplitude,Szkely2004TESTINGFE}.
We refer the readers to \cite{villani_topics_2003,villani_optimal_2009} for comprehensive theoretical investigations. 

In the discrete setting, Kantorovich's relaxation (cf.\cref{eq:DOT-X}), which is also known as the Monge--Kantorovich problem \cite{brezis_remarks_2018}, seeks an optimal transfer plan (an $n$-by-$n$ nonnegative matrix) that minimizes the total transport cost between two given mass distributions in the $n$-dimensional probability simplex. Except for this classical formulation, nowadays, there are some extensions, such as partial optimal mass transport \cite{caffarelli_free_2010} and capacity-constrained transport problem \cite{korman_optimal_2014}. All these transport-like programmings share the common feature of marginal constraint and can be formulated as standard but large scale linear programmings (LP); see \cref{sec:prob} for more details. Besides, entropy regularization, i.e., the logarithmic barrier function, has been used to relax the  nonnegative restriction and provides an approximate optimization problem that possesses some nice properties including strong convexity of the primal form (which corresponds to the smoothness of the dual problem)  and closed projection onto the marginal constraint.

Let us first review some existing methods based on entropy regularization. The well-known Sinkhorn algorithm \cite{cuturi_sinkhorn_2013,sinkhorn_diagonalequivalenceto_1967} and its greedy adaptation, called Greenhorn \cite{altschuler_near-linear_2017}, are fixed-point type iterations. Sinkhorn's algorithm was proved to converge linearly (cf. \cite[Theorem 35]{merigot_optimal_2021}) but the rate
is exponentially degenerate with respect to the regularization parameter, and the theoretical complexity bound $\mathcal O(n^2/\epsilon^2)$ can be found in \cite{dvurechensky_computational_2018,lin_efficient_2019}. Here, we mention that the iterative Bregman projection \cite{benamou_iterative_2015} with Kullback--Leibler divergence is equivalent to Sinkhorn's algorithm.
By virtue of the smoothness of the dual problem, accelerated mirror descent methods have been proposed in \cite{dvurechensky_computational_2018,lin_efficient_2019}, and the provable complexity is $\mathcal O(n^{2.5}\sqrt{\ln n}/\epsilon)$. For more methods using the entropy regularization, we refer to 
\cite{chernov_fast_2016,gasnikov_efficient_2016,guminov_accelerated_2021}, and one can also consult 
\cite{benamou_optimal_2021,merigot_optimal_2021,Peyre2019} on quite complete surveys about numerical methods. 

It is worth noticing that, the solution to the entropy regularized problem exists uniquely, and as the regularization parameter vanishes, it converges (exponentially) to an optimal transport plan with maximal entropy among all the optimal plans; see \cite{cominetti_asymptotic_1994} and \cite[Proposition 4.1]{Peyre2019}. 
However, practically, one cannot choose arbitrarily small parameters due to the round-off issue and instability effect. That being said, the log-domain technique \cite{chizat_scaling_2018} enhances the stability, and both vectorization and parallelization can be applied to Sinkhorn's algorithm. 

The complexity of matrix-vector multiplication in each iteration of the Sinkhorn algorithm is $\mathcal O(n^2)$ for general cases. If the transport cost function enjoys a separable structure with $d$ blocks, then efficient implementation achieves the reduced cost $\mathcal O(n^{1+1/d})$; see \cite[Section 4.3]{Peyre2019}. More recently, for translation invariance cost functions (for instance, the Euclidean distance) that imply the cyclic property, Liao et al. \cite{liao_fast_2022} proposed an optimal $\mathcal O(n)$ algorithm. Hence, to provide an {\it approximate} optimal transport plan with moderate regularization parameter, entropy-based methods are efficient, especially for some cost functions with nice properties. This makes them very popular in real applications, especially for computing Wasserstein distances between histograms.

On the other hand, augmented Lagrangian method (ALM) and alternating direction method of multipliers (ADMM) can be applied to transport-like problems as well. By the celebrated stability result of linear inequality system \cite{mangasarian_lipschitz_1987,Robinson1973}, we have global linear convergence for ADMM \cite{eckstein_alternating_1990}. For general convex objectives, the provable nonergodic rate of many accelerated variants of (linearized) ALM and ADMM  is $\mathcal O(1/k)$; see 
\cite{Li2019,luo_accelerated_pd_2021,luo_unified_2021,Xu2017}.
The method \cite[Algorithm 1]{Xu2017} possesses a faster sublinear rate $\mathcal O(1/k^2)$ but involves a large scale quadratic programming (of dimension $n^2$) of the primal variable.

There are also classical LP solvers such as the interior-point method \cite{lee2014path,pele2009fast} and semismooth Newton-based algorithms \cite{bai_computing_2007,brauer_sinkhorn-newton_2018,li_asymptotically_2020,liu_multiscale_2022}. These methods have to solve a symmetric positive definite (SPD) system per (inner) iteration, and prevailing linear solvers are (sparse) Cholesky decomposition and preconditioned conjugate gradient (PCG). However, the corresponding SPD system might be nearly singular and ill-conditioned as the problem size increases, and thus the number of iterations grows dramatically. Therefore, efficient and robust linear solvers play important roles in these algorithms. We also refer the readers to \cite{ahuja_network_1993,Bertsekas_auction_1992,Peyre2019} for some combinatorial methods.
\subsection{Main results}
In this work, we propose an efficient inexact primal-dual method for the generalized transportation problem (cf.\cref{eq:GOT-X}), which includes a large class of transport-like programmings, such as optimal mass transport, partial optimal transport, and capacity-constrained transport problem. Our algorithm is based on proper time discretization of the accelerated primal-dual dynamical system \cite{luo_accelerated_pd_2021} and adopts the semismooth Newton (SsN) iteration as the inner solver. 

In the setting of inexact computations, we prove the contraction estimate
\[
\mathcal E_{k+1}-\mathcal E_k\leq -\alpha_{k}\mathcal E_{k+1},
\]
where $\alpha_k>0$ is the step size and $\mathcal E_{k}$ denotes the discrete Lyapunov function (cf.\cref{eq:Ek}). Besides, we establish the global convergence rate of the objective residual and the feasibility violation (see \cref{thm:conv-Inexact-PD}):
\[
\snm{	h(x_k)-h(x^*) }+ \nm{Gx_k+I_Yy_k+I_Zz_k -b}\leq C \prod_{i=0}^{k-1}\frac{1}{1+\alpha_i}.
\]
This implies linear rate as long as $\alpha_k\geq \alpha>0$ and superlinear convergence follows if $\alpha_k\to+\infty$. See \cref{rem:ak-linear,rem:ak-sublinear} for detailed discussions.

The inner problem (cf.\cref{eq:prox-lk1}) is a nonlinear equation with strongly semismooth property. This motivates us to adopt the SsN iteration, which requires solving an $N$-by-$N$ linear SPD system
\begin{equation}\label{eq:Txi}
	\mathcal T\xi=\left(	\epsilon I+ T\Lambda T^\top\right)\xi = z,
\end{equation}
where $\Lambda$ is diagonal and $\epsilon>0$ is a small number. Utilizing the hidden structure of  $\mathcal T$, we transfer it equivalently into $A = \epsilon D+A_0$, where $D$ is diagonal and $A_0$ is the Laplacian matrix of a bipartite graph, and then develop a robust algebraic multigrid (AMG) algorithm. Invoking the well-known Xu--Zikatanov identity \cite{xu_method_2002}, we prove the convergence rate of the two-level case (see Theorem \ref{thm:conv-2g-AMG}):
\begin{equation*}
	\nm{\xi_{i+1}-\xi^*}_{A}= \left(1-\frac{1}{c_1}\right)\nm{\xi_{i}-\xi^*}_{ A}
	,\quad\,c_1 \leq C+\frac{1}{1-\nm{I-\bar R_c A_c}_{A_c}},
\end{equation*}
where $\bar R_c$ and $A_c$ are respectively the coarse level solver and coarse level matrix, and $C\geq 1$ is independent of the number $\epsilon$ and the problem size $N$. 
\subsection{Outline}
The rest of this paper is organized as follows. In \cref{sec:prob}, we give the problem setting and introduce the generalized transportation problem. After that, in \cref{sec:SsN-PD}, we present our inexact SsN-based primal-dual method and prove the global convergence rate via a discrete Lyapunov function. Then, in \cref{sec:near-singu-SPD}, we focus on the linear SPD system arising from the SsN iteration and transfer it into an equivalent graph Laplacian system, for which a robust and efficient AMG algorithm is proposed and analyzed in \cref{sec:amg,sec:conv-2gd}, respectively. We provide several numerical tests in \cref{sec:num} to show the robustness of the AMG algorithm and the performance of our overall SsN-AMG-based inexact primal-dual method. Finally, we conclude our work in \cref{sec:con}.

\section{Problem Setting}
\label{sec:prob}
In this part, we list several typical transport-like programmings arising from either mathematical extensions or practical applications. Those problems share the common feature of marginal constraint and will be treated in a unified way.
\subsection{Transport-like problems}
\label{sec:prob-exam}
\subsubsection{Optimal transport}
In the setting of optimal mass transportation  \cite{brezis_remarks_2018,kantorovich_translocation_1942}, we are given a cost matrix $C\in\R^{m\times n}_+$ and two vectors $\mu\in\R^n_+,\,\nu\in\R_+^m$ satisfying the mass conservation condition: ${\bf 1}_n^{\top}\mu={\bf 1}_m^{\top}\nu$, and aim to solve the minimization problem
\begin{equation}\label{eq:DOT-X}
	\min_{X\in \mathcal B(\mu,\nu)}\,\dual{C,X}:=\sum_{i=1}^{m}\sum_{j=1}^nC_{ij}X_{ij},
\end{equation}
where 
$\mathcal B(\mu,\nu): = \left\{X\in\R^{m\times n}_+:\,X^{\top}{\bm 1}_m = \mu,\,X{\bm 1}_n= \nu \right\}$ denotes the {\it transportation polytope}, with ${\bm 1}_n({\bf 1}_m)\in\R^n(\R^m)$ being the vector of all ones. According to \cite[Chapter 8]{Brualdi2006}, $\mathcal B(\mu,\nu)$ is nonempty, convex and bounded. It follows immediately from \cite[Corollary 2.3]{Bertsimas1997} that \cref{eq:DOT-X} admits at least one solution, which is called an {\it optimal transport plan}.

When $m=n$ and $\mu = \nu = {\bf 1}_n$, the transportation polytope coincides with the Birkhoff polytope $\mathcal B_n := \{X\in\R_+^{n\times n}:X^{\top}{\bf 1}_n= {\bf 1}_{n},\,X{\bf 1}_n ={\bf 1}_n\}$, which consists of all $n$-by-$n$ doubly stochastic matrices. The celebrated Birkhoff--Von-Neumann theorem (cf. \cite[Theorem 17]{merigot_optimal_2021}) states that $\mathcal B_n$ is the convex hull of all permutation matrices. Therefore, it has frequently been used to relax some combinatorial or nonconvex problems \cite{fogel_convex_2015}. In particular, the optimal transport \cref{eq:DOT-X} is exactly the convex relaxation of the linear assignment problem \cite{burkard_assignment_2009}.
\subsubsection{Birkhoff projection} 
Given any $\Phi\in\R^{n\times n}_+$, the Birkhoff projection in terms of the Frobenius norm $\nm{\cdot}_F$ considered in \cite{khoury_closest_1998,li_efficient_2020} reads as
\begin{equation}\label{eq:proj-Bn}
	\min_{X\in\mathcal B_n}\frac{1}{2}\nm{X-\Phi}_F^2,
\end{equation}
which actually seeks the nearest doubly matrix of $\Phi$ and usually arises from the relaxations of some nonconvex programmings \cite{fogel_convex_2015,jiang_l_p-norm_2016,Lim_beyond_2014}. Besides, in the setting of numerical simulation for circuit networks \cite{Bai_partial_2001}, we have to fix some components:
\begin{equation}
	\label{eq:Xij}
	X_{ij} = 
	\Phi_{ij},\quad\forall\,i\in \mathcal I,\,j\in\mathcal J,
\end{equation}
where $\mathcal I,\,\mathcal J\subset\{1,2,\cdots,n\}$ are two given index sets, and this leads to the problem of finding the best approximation in the Birkhoff polytope $\mathcal B_n$ with prescribed entry constraint \cite{bai_computing_2007,glunt_nearest_1998}.
\subsubsection{Partial optimal transport}
In standard optimal transport \cref{eq:DOT-X}, the marginal distributions $\mu$ and $\nu$ are required to have the same mass. Mathematically, this is quite restrictive and practically, the {\it unbalanced} case ${\bf 1}_n^{\top}\mu\neq{\bf 1}_m^{\top}\nu$ stems from the positive-unlabeled learning \cite{chapel_partial_2020} and the representation of dynamic meshes for controlling the volume of objects with free boundaries \cite{levy_partial_2022}.

This leads to a problem called partial optimal transport \cite{benamou_iterative_2015,chizat_scaling_2018}. More precisely, we aim to transport only a given fraction of mass $a\in(0,a_{\max}]$ where $a_{\max}:=\min\{{\bf 1}_n^{\top}\mu,\,{\bf 1}_m^{\top}\nu\}$, and minimize the total cost 
\begin{equation}\label{eq:PDOT}
	\min_{X\in\R_+^{m\times n}} \dual{C,X}\quad\st\,X^{\top}{\bf 1}_m\leq \mu,\quad X{\bf 1}_n \leq \nu,\quad {\bf 1}_m^\top X{\bf 1}_n = a.
\end{equation}
When $a = \mu^\top{\bf 1}_n=\nu^\top{\bf 1}_m$, this amounts to optimal transport \cref{eq:DOT-X}. Well-posedness of \cref{eq:PDOT} (in the continuous setting) was established in \cite{caffarelli_free_2010} and extended by Figalli \cite{figalli_optimal_2010}. 
\subsection{Generalized transportation problem}
Clearly, the transport plan $X$ belongs to a box region $\mathcal K:=\{X\in\R^{m\times n}:\Theta\leq X\leq \Gamma\}$, where $\Theta,\,\Gamma\in\R^{m\times n}$ with $0\leq \Theta_{ij}<\infty$ and $\Theta_{ij}\leq \Gamma_{ij}\leq\infty$. Introduce two slack variables $y\in\R^n$ and $z\in\R^m$, together with their feasible regions $\mathcal Y = \R^{n}_+$ (or $\{0\}$) and $\mathcal Z = \R^{m}_+$ (or $\{0\}$). This allows us to include the unbalanced case ${\bf 1}_n^{\top}\mu\neq{\bf 1}_m^{\top}\nu$. We also impose the total mass constraint $\pi(X) = a$, where $a\in(0,\min\{{\bf 1}_n^{\top}\mu,\,{\bf 1}_m^{\top}\nu\}]$  and $\pi:\R^{m\times n}\to\R^r$ is a linear operator. 

Then the generalized transportation problem reads as follows
\begin{equation}\label{eq:GOT-X}
	\min_{(X,y,z)\in\Omega}H(X)\quad\st X^\top {\bf 1}_m +y=\mu,\quad X{\bf 1}_n  + z =\nu,\quad \pi(X) = a, 
\end{equation} 
where $\Omega:=\mathcal K\times \mathcal Y\times \mathcal Z$ and $H(X):=\sigma/2\nm{X-\Phi}_{F}^2+\dual{C,X}$ with $\sigma\geq 0$ and $\Phi\in\R_+^{m\times n}$. This generic formulation contains all problems mentioned previously in \cref{sec:prob-exam}, and also includes other transport-like problems such as capacity constrained optimal transport \cite{korman_optimal_2014} and the machine loading problem \cite{Ferguson_1956}. Throughout this paper, assume \cref{eq:GOT-X} exists at least one solution $(X^*,y^*,z^*)\in\Omega^*$, where $\Omega^*$ is the set of all global minimizers.

As usual, denote by $\vect{\times}$ the vector expanded by the matrix $\times $ by column. Let $h(x) :=  \sigma/2\nm{x-\phi}^2+ c^\top x$ with $c = \vect{C}$ and $\phi = \vect{\Phi}$, and introduce  $\mathcal X: = \{x\in\R^{mn}:\theta_i\leq x_i\leq \gamma_i,\,1\leq i\leq mn\}$ with $\theta=\vect{\Theta}$ and $\gamma = \vect{\Gamma}$ . Suppose the linear operator $\pi$ admits a matrix representation $\Pi\in\R^{r\times mn}$ such that $\pi(X) = \Pi\vect{X} $ for all $X\in \R^{m\times n}$. 
Then, we rearrange \cref{eq:GOT-X} as a standard affine constrained optimization problem:
\begin{equation}\label{eq:GOT-x}
	\min_{(x,y,z)\in\Sigma}\,h(x)\quad\st Gx+I_Yy+I_Zz =b,
\end{equation}
where $\Sigma:=\mathcal X\times \mathcal Y\times \mathcal Z$  and
\begin{equation}\label{eq:A-b}
	G := \begin{pmatrix}
		T\\\Pi
	\end{pmatrix},\quad
	T: = \begin{pmatrix}
		I_n\otimes {\bf 1}_m^\top \\
		{\bf 1}_n^\top\otimes 	I_m \\
	\end{pmatrix},\quad 
	I_Y: = \begin{pmatrix}
		I_n\\
		O\\
		O
	\end{pmatrix},\quad 
	I_Z: = \begin{pmatrix}
		O\\I_m\\O
	\end{pmatrix},\quad 
	b := \begin{pmatrix}
		\mu\\\nu\\a
	\end{pmatrix}.
\end{equation}
Clearly, $ (X^*,y^*,z^*)\in\Omega^*$ if and only if $ (x^*,y^*,z^*)\in\Sigma^*$ and $x^* = \vect{X^*}$, where $\Sigma^*$ denotes the set of all minimizers of problem \cref{eq:GOT-x}.

At the end of this section, let us make some conventions. The angle bracket $\dual{\cdot,\cdot}$ stands for the usual Euclidean inner product of two vectors.
For any SPD matrix $A$, define the $A$-inner product $\dual{\cdot,\cdot}_A:= \dual{A\cdot,\cdot}$ and the induced $A$-norm $\nm{\cdot}_{A} = \sqrt{\dual{\cdot,\cdot}_A}$. When no confusion arises, the $A$-norm of a matrix $B$ is denoted by $\nm{B}_A:=\sup_{x\neq 0}\nm{Bx}_A/\nm{x}_A$. 
The proximal mapping $\prox_{\eta g}:\R^n\to\R^n$ of a properly closed convex function $g$ with $\eta>0$ is
\[
\prox_{\eta g}(x): = \mathop{\argmin}_{y\in\R^n}\left\{g(y) + \frac{1}{2\eta}\nm{y-x}^2\right\}\quad\forall\, x\in\R^n.
\]
Let $\delta_{\mathcal O}(\cdot)$ be the indicator function of a nonempty closed convex subset $\mathcal O\subset \R^n$ and define the normal cone at $x$ as follows
\[
\mathcal N_{\mathcal O}(x):=\left\{z\in\R^n:\dual{z,x-y}\geq 0\quad\forall\,y\in\mathcal O\right\}.
\]
For simplicity, we also write $\proj_{\mathcal O}=\prox_{\eta\delta_\mathcal O}$ for all $\eta>0$.

\section{An Inexact SsN-based Primal-Dual Method}
\label{sec:SsN-PD}
For any $\lambda\in\R^{m+n+r}$ and $\bm u=(x,y,z)\in\R^{mn}\times\R^n\times\R^m$, define the Lagrangian function for \cref{eq:GOT-x}:
\begin{equation}\label{eq:L2}
	\mathcal L(\bm u,\lambda) : = h(x)+\delta_{\Sigma}(\bm u)+\dual{\lambda,H\bm u-b},
\end{equation}
where $H = (G,I_Y,I_Z)$.
Notice that $\mathcal L(\cdot,\lambda)$ is convex and we set $	\partial_{\bm u}\mathcal L(\bm u,\lambda): = \nabla h(\bm u)+\mathcal N_{\Sigma}(\bm u)+ H^\top\lambda$,
where $	\nabla h(\bm u) = \sigma (x,{\bf 0}_n,{\bf 0}_m)-\widetilde{c}$ with $\widetilde{c} = (\sigma \phi-c,{\bf 0}_{n},{\bf 0}_m)$. 
\subsection{An inexact primal-dual algorithm}
\label{sec:PDM}
Let us start from the accelerated primal-dual flow dynamics proposed in \cite{luo_accelerated_pd_2021}:
\begin{equation}\label{eq:apd}
	\left\{
	\begin{aligned}
		\beta\lambda'={}&\nabla_\lambda \mathcal{L}(\bm v,\lambda),\\
		\bm u'={}&\bm v-\bm u,\\
		\beta \bm v' \in{}&-\partial_{\bm u}\mathcal L(\bm u,\lambda),
	\end{aligned}
	\right.
\end{equation}
where $\beta(t) = e^{-t}$ is a built-in time rescaling factor.
For the smooth case, i.e., $\Sigma$ is the entire space, we have exponential decay property (cf. \cite[Lemma 2.1]{luo_accelerated_pd_2021})
\[
\mathcal{L}(\bm u(t),\lambda^*)-\mathcal{L}(\bm u^*,\lambda(t))
+\snm{h(x(t))-h(x^*)}+\nm{H\bm u(t)-b}\leq C e^{-t}.
\]
In this work, we leave the well-posedness and exponential decay of the differential inclusion \cref{eq:apd} alone but focus on its implicit Euler discretization.

More precisely, given the current iteration
$(\bm u_k,\bm v_k,\lambda_k)$, compute
$(\bm u_{k+1},\bm v_{k+1},\lambda_{k+1})$ by that
\begin{subnumcases}{	\label{eq:apd-im-x-im-l}}
	\beta_k \frac{\lambda_{k+1}-\lambda_k}{\alpha_k} = {} \nabla_\lambda \mathcal L(\bm v_{k+1},\lambda_{k+1}),\label{eq:apd-im-x-im-l-l}\\
	\frac{\bm u_{k+1}-\bm u_{k}}{\alpha_k}={}\bm v_{k+1}-\bm u_{k+1},\label{eq:apd-im-x-im-l-x}\\
	\beta_k \frac{\bm v_{k+1}-\bm v_k}{\alpha_k} \in -\partial_{\bm u}\mathcal L\left(\bm u_{k+1}, \lambda_{k+1}\right),
	\label{eq:apd-im-x-im-l-v}
\end{subnumcases}
where $\alpha_k>0$ denotes the step size and the parameter sequence $\{\beta_k\}_{k\in\mathbb N}$ is updated by 
\begin{equation}\label{eq:thetak}
	\beta_{k+1}-\beta_{k} = -\alpha_k\beta_{k+1},\quad \beta_0=1.
\end{equation}
By \eqref{eq:apd-im-x-im-l-x}, we replace $\bm v_{k+1}$ by $\bm u_{k+1}$ and then put it into  \eqref{eq:apd-im-x-im-l-l} and \eqref{eq:apd-im-x-im-l-v} to obtain
\begin{subnumcases}{	\label{eq:lk1-uk1}}
	\label{eq:lk1}
	\beta_{k+1}\lambda_{k+1}={}H\bm u_{k+1}+\widetilde{\lambda}_k,\\
	\label{eq:uk1}
	D_{k}\bm u_{k+1}\in  \bm w_{k} - H^\top\lambda_{k+1}-\mathcal N_{\Sigma}(\bm u_{k+1}),
\end{subnumcases}
where $	D_{k} ={\rm diag}(\eta_{k}I_{mn},\tau_{k}I_n,\tau_{k}I_m)$ and 
\[
\left\{
\begin{aligned}
	\bm w_k={}&  \widetilde{c}+\beta_{k}(\bm u_k+\alpha_k\bm v_k)/\alpha_k^2,\\
	\eta_{k} ={}&\sigma+\tau_{k},\,\tau_{k} = \beta_k(1+\alpha_k)/\alpha_k^2,\\
	\widetilde{\lambda}_k={}&\beta_{k+1}\left[\lambda_k-\beta_{k}^{-1}(H\bm u_k- b)\right]-b.
\end{aligned}
\right.
\]

From \eqref{eq:uk1} we have $	\bm u_{k+1} ={} \proj_{\Sigma}\!\left(D_{k}^{-1}(\bm w_k-
H^\top\lambda_{k+1})\right)$. Plugging  this into \eqref{eq:lk1} gives a nonlinear equation
\begin{equation}\label{eq:prox-lk1}
	F_k(\lambda_{k+1}) = 0,
\end{equation}
where the mapping $F_k:\R^{m+n+r}\to\R^{m+n+r}$ is defined by
\begin{equation}\label{eq:Fk}
	F_k(\lambda): =	\beta_{k+1}\lambda-H\proj_{\Sigma}\left(D_{k}^{-1}(\bm w_k-
	H^\top\lambda)\right)-\widetilde{\lambda}_k\quad\forall\,\lambda\in\R^{m+n+r}.
\end{equation}
It is well-known that $\proj_{\Sigma}$ is monotone and $1$-Lipschitz continuous (cf. \cite[Proposition 12.27]{Bauschke2011}).
In \cref{sec:SsN}, we will see that \cref{eq:prox-lk1} is nothing but the Euler--Lagrange equation for minimizing a smooth and strongly convex objective (see \cref{eq:cal-Fk}). Therefore it admits a unique solution which is denoted by $\lambda_{k+1}^{\#}$ (instead of $\lambda_{k+1}$), and we obtain 
\begin{equation}\label{eq:uk1-true}
	\bm u_{k+1}^{\#} ={} \proj_{\Sigma}\!\left(D_{k}^{-1}(\bm w_k-
	H^\top\lambda_{k+1}^{\#})\right),\quad 	\bm v_{k+1}^{\#} = \bm u_{k+1}^{\#}+\frac{\bm u_{k+1}^{\#}-\bm u_{k}}{\alpha_k}.
\end{equation}
This means $(\bm u_{k+1}^{\#},\bm v_{k+1}^{\#},\lambda_{k+1}^{\#})$ is the exact solution to the implicit scheme \cref{eq:apd-im-x-im-l} at the $k$-th step.

In practical computation, however, the inner problem \cref{eq:prox-lk1} is often solved approximately. 
Below, an inexact version of \cref{eq:apd-im-x-im-l} is summarized in \cref{algo:Inexact-PD}. Then in \cref{sec:conv,sec:SsN}, we will present the convergence analysis and apply the semi-smooth Newton iteration (\cref{algo:SsN}) to solve the nonlinear equation \cref{eq:prox-lk1}. 
\begin{algorithm}[H]
	\caption{Inexact Primal Dual Method for \cref{eq:GOT-x}}
	\label{algo:Inexact-PD}
	\begin{algorithmic}[1] 
		\REQUIRE  $\beta_{0}=1,\,\bm v_0,\,\bm u_0 \in\R^{mn}\times\R^n\times \R^m$ and $ \lambda_0\in\R^{m+n+r}$.
		\FOR{$k=0,1,\ldots$}
		\STATE Choose the step size $\alpha_k>0$ and the tolerance $\epsilon_k>0$.
		\STATE Set $	\tau_{k} ={} \beta_k(1+\alpha_k)/\alpha_k^2$ and $\eta_{k} =\sigma+\tau_{k}$.
		\STATE Set $D_{k} ={\rm diag}(\eta_{k}I_{mn},\tau_{k}I_n,\tau_{k}I_m)$ and $\bm w_k={} \widetilde{c}+\beta_{k}(\bm u_k+\alpha_k\bm v_k)/\alpha_k^2$.		
		\STATE Update $\displaystyle \beta_{k+1}= \beta_k/(1+\alpha_k)$ and set $	\widetilde{\lambda}_k={}\beta_{k+1}\left[\lambda_k-\beta_{k}^{-1}(H\bm u_k- b)\right]-b$.
		\STATE Apply \cref{algo:SsN} to \cref{eq:prox-lk1} to obtain $\lambda_{k+1}$ such that 
		$\|\lambda_{k+1}-\lambda_{k+1}^{\#}\|\leq 
		\epsilon_k$.\label{line:lk1}
		\STATE Update $	\bm u_{k+1} = \proj_{\Sigma}\left(D_{k}^{-1}(\bm w_k-
		H^\top\lambda_{k+1})\right)$.
		\STATE Update $	\bm v_{k+1} = \bm u_{k+1}+(\bm u_{k+1}-\bm u_k)/\alpha_k$.		
		\ENDFOR
	\end{algorithmic}
\end{algorithm}

According to \cref{thm:conv-Inexact-PD}, the convergence rate is related to the step size $\alpha_k$ and we provide detailed discussions in \cref{rem:ak-linear,rem:ak-sublinear}. In addition, the stop criterion in step
 \ref{line:lk1} of \cref{algo:Inexact-PD} is convenient for the upcoming convergence rate proof but not practical as it requires the true solution $\lambda^{\#}_{k+1}$. In numerical experiments, we focus on the quantity  $\nm{F_k(\lambda_{k+1})}$, which provides a computable posterior indicator.
\subsection{Rate of convergence}
\label{sec:conv}
Let $\{(\beta_{k+1},\bm u_{k+1},\bm v_{k+1},\lambda_{k+1})\}_{k\in\mathbb N}$ be generated by \cref{algo:Inexact-PD} with $\{\alpha_k\}_{k\in\mathbb N}$ and $\{\epsilon_k\}_{k\in\mathbb N}$. It is clear that $\{\bm u_{k}\}_{k\in\mathbb N}\subset \Sigma$. Following \cite{luo_accelerated_pd_2021,luo_unified_2021,luo_primal-dual_2022}, introduce a discrete Lyapunov function
\begin{equation}
	\label{eq:Ek}
	\mathcal E(\beta_k,\bm u_k,\bm v_k,\lambda_k): = \mathcal L(\bm u_k,\lambda^*)-\mathcal L(\bm u^*,\lambda_k)
	+\frac{\beta_k}{2}\left(\nm{\bm v_k-\bm u^*}^2+\nm{\lambda_k-\lambda^*}^2\right),
\end{equation}
and for simplicity, we write $	\mathcal E_k=	\mathcal E(\beta_k,\bm u_k,\bm v_k,\lambda_k)$.

Let $k\in\mathbb N,\,(\beta_k,\bm u_k,\lambda_k)$ and $(\alpha_k,\epsilon_k)$ be given. By \cite[Theorem 3.1]{luo_accelerated_pd_2021}, we have the one-step estimate
\begin{equation}\label{eq:diff-Ek-}
	\mathcal E^{\#}_{k+1}-	\mathcal E_{k}
	\leq -\alpha_k		\mathcal E^{\#}_{k+1},
\end{equation}
where $	\mathcal E^{\#}_{k+1}: =\mathcal E(\beta_{k+1},\bm u_{k+1}^{\#},\bm v_{k+1}^{\#},\lambda_{k+1}^{\#})$ and $(\bm u_{k+1}^{\#},\bm v_{k+1}^{\#},\lambda_{k+1}^{\#})$ is the exact solution to the implicit Euler discretization \cref{eq:apd-im-x-im-l} at the $k$-th iteration. This also implies that
\begin{equation}\label{eq:diff-Ek1-Ek}
	\mathcal E_{k+1}-\frac{\mathcal E_{k}}{1+\alpha_k}  ={}\mathcal E_{k+1}-\mathcal E_{k+1}^{\#}+\mathcal E_{k+1}^{\#}-\frac{\mathcal E_{k}}{1+\alpha_k} 
	\leq {}\mathcal E_{k+1}-\mathcal E_{k+1}^{\#},
\end{equation}
which leads to the following one-iteration estimate.

\begin{lem}\label{lem:bd-Ek1}
	Let $(\beta_{k+1},\bm u_{k+1},\bm v_k,\lambda_{k+1})$ be the output of the $k$-th iteration of \cref{algo:Inexact-PD} with $(\beta_k,\bm u_k,\lambda_k)$ and $(\alpha_k,\epsilon_k)$. Then we have
	\begin{equation}\label{eq:bd-Ek1}\small
		\begin{aligned}
			\mathcal{E}_{k+1}\leq {}&\frac{\mathcal{E}_k}{1+\alpha_k}
			+\epsilon_k\beta^{-1}_{k}\alpha_k^2\nm{H}\left(\sigma\nm{x^*}+\nm{H^\top\lambda^*-\widetilde{c}}\right)\\
			{}&\quad
			+\epsilon_k\beta_k^{-1}
			\left(\sigma \alpha_k^2\nm{H}\nm{x_{k+1}-x^*}
			+\alpha_k\beta_k\nm{H}\nm{\bm v_{k+1}-\bm u^*}
			+\beta_{k}^2\nm{\lambda_{k+1}-\lambda^*}
			\right).
		\end{aligned}
	\end{equation}
\end{lem}
\begin{proof}	
	Thanks to \cref{eq:diff-Ek1-Ek}, it is sufficient to focus on the difference
	\begin{equation}\label{eq:diff-E2}\small
		\begin{aligned}
			\mathcal E_{k+1}-\mathcal E_{k+1}^{\#}={}&\mathcal L(\bm u_{k+1},\lambda^*)-\mathcal L(\bm u^{\#}_{k+1},\lambda^*)+\frac{\beta_{k+1}}{2}\big(\nm{\bm v_{k+1}-\bm u^*}^2
			-\|\bm v^{\#}_{k+1}-\bm u^*\|^2\big)\\
			&\qquad +\frac{\beta_{k+1}}{2}\big( \nm{\lambda_{k+1}-\lambda^*}^2-\|\lambda^{\#}_{k+1}-\lambda^*\|^2\big).
		\end{aligned}
	\end{equation}
	Since $\proj_{\Sigma}$ is 1-Lipschitz continuous and
	\[\begin{aligned}
		\bm u_{k+1} ={}\proj_{\Sigma}\left(D_{k}^{-1}(\bm w_k-
		H^\top\lambda_{k+1})\right)
		,\quad
		\bm u^{\#}_{k+1} ={}&\proj_{\Sigma}\big(D_{k}^{-1}(\bm w_k-
		H^\top\lambda_{k+1}^\#)\big),
	\end{aligned}
	\]
	it follows from the fact $\nm{D_k^{-1}}\leq 1/\tau_{k}$ that
	\begin{equation}\label{eq:uk1-est}
		||\bm u_{k+1}-\bm u^{\#}_{k+1}||\leq \frac{1}{\tau_k} ||H^\top(\lambda_{k+1}-\lambda^{\#}_{k+1})||\leq \frac{\epsilon_k\beta^{-1}_{k}\alpha_k^2}{1+\alpha_k}\nm{H}.
	\end{equation}
	By \cref{eq:uk1-true}, we have $	\bm v_{k+1}^\# = \bm u_{k+1}^\#+(\bm u_{k+1}^\#-\bm u_k)/\alpha_k$, which together with the update for $\bm v_{k+1}$ in \cref{algo:Inexact-PD} yields the identity $\bm v_{k+1}-\bm v^{\#}_{k+1}=(1+1/\alpha_k)(\bm u_{k+1}-\bm u^{\#}_{k+1})$. Hence it holds that
	\[
	||\bm v_{k+1}-\bm v^{\#}_{k+1}||\leq \frac{\alpha_k+1}{\alpha_k}||\bm u_{k+1}-\bm u^{\#}_{k+1}||\leq \epsilon_k\beta^{-1}_{k}\alpha_k\nm{H},
	\]
	which gives
	\[\begin{aligned}
		\nm{\bm v_{k+1}-\bm u^*}^2
		-\|\bm v^{\#}_{k+1}-\bm u^*\|^2=		{}&2\langle \bm v_{k+1}-\bm u^*,\bm v_{k+1}-\bm v^{\#}_{k+1}\rangle-\|\bm v_{k+1}-\bm v^{\#}_{k+1}\|^2\\
		\leq{}&
		2\epsilon_k\beta^{-1}_{k}\alpha_k\nm{H}\nm{\bm v_{k+1}-\bm u^*}.
	\end{aligned}
	\]
	Similarly, we have
	\[\begin{aligned}
		{}& \nm{\lambda_{k+1}-\lambda^*}^2-\|\lambda^{\#}_{k+1}-\lambda^*\|^2\leq 
		2\epsilon_k\nm{\lambda_{k+1}-\lambda^*}.
	\end{aligned}
	\]
	Plugging the above two estimates into \cref{eq:diff-E2} implies
	\begin{equation}\label{eq:diff-Ek1}
		\mathcal E_{k+1}-\mathcal E_{k+1}^{\#}\leq{}\mathcal L(\bm u_{k+1},\lambda^*)-\mathcal L(\bm u^{\#}_{k+1},\lambda^*)+\epsilon_k\alpha_k\nm{H}\nm{\bm v_{k+1}-\bm u^*}+\epsilon_k\beta_{k}\nm{\lambda_{k+1}-\lambda^*},
	\end{equation}
	where we used the relation $\beta_{k+1}\leq \beta_k$.
	
	To the end, let us estimate the first difference term in \cref{eq:diff-Ek1} as follows. It is clear that 
	\[
	h(x_{k+1})-h(x_{k+1}^{\#})=\sigma\langle x_{k+1}-x^{\#}_{k+1}, x_{k+1}\rangle
	- \langle \widetilde{c},\bm u_{k+1}-\bm u^{\#}_{k+1}\rangle-\frac{\sigma}{2}\|x_{k+1}-x^{\#}_{k+1}\|^2.
	\]
	Invoking \cref{eq:uk1-est} and the fact $\bm u^{\#}_{k+1},\,\bm u_{k+1}\in \Sigma$, we find 
	\[
	\begin{aligned}
		\mathcal L(\bm u_{k+1},\lambda^*)-\mathcal L(\bm u^{\#}_{k+1},\lambda^*) ={}& h(x_{k+1})-h(x_{k+1}^{\#}) + \langle H^\top\lambda^*,\bm u_{k+1}-\bm u^{\#}_{k+1}\rangle\\
		\leq {}&\left(\sigma\nm{x_{k+1}}+\nm{H^\top\lambda^*-\widetilde{c}}\right)\|\bm u_{k+1}-\bm u^{\#}_{k+1}\|\\
		\leq{}& \epsilon_k\beta^{-1}_{k}\alpha_k^2\nm{H}
		\left(\sigma\nm{x_{k+1}}+\nm{H^\top\lambda^*-\widetilde{c}}\right).
	\end{aligned}
	\]
Combining this with \cref{eq:diff-Ek1} and
	the triangle inequality $\nm{x_{k+1}}\leq \nm{x^*}+\nm{x_{k+1}-x^*}$, we obtain \cref{eq:bd-Ek1} and complete the proof of this lemma.
\end{proof}

From \cref{eq:thetak} we obtain $\beta_{k} = \prod_{i=0}^{k-1}\frac{1}{1+\alpha_i}$. To derive the concrete convergence rate of \cref{algo:Inexact-PD},
let us introduce 
\[
\widehat{\varepsilon}_k := \sum_{i=0}^{k-1}\epsilon_i\alpha_i^2\beta_{i}^{-2}
,\quad 
\widetilde{\varepsilon}_k :=\sum_{i=0}^{k-1}\epsilon_i\beta_i^{-3/2}(\alpha_i^2+\alpha_i\sqrt{\beta_i}+\beta_{i}^{3/2}),\quad k\geq 1,
\]
and for $k=0$, set $\widehat{\varepsilon}_0 = \widetilde{\varepsilon}_0 = 0$. 
\begin{thm}\label{thm:conv-Inexact-PD}
	Let $\{\bm u_k\}_{k\in\mathbb N}=\{(x_k,y_k,z_k)\}_{k\in\mathbb N}\subset\Sigma$ and $\{\lambda_k\}_{k\in\mathbb N}$ be generated by \cref{algo:Inexact-PD} with arbitrary step size sequence $\{\alpha_k\}_{k\in\mathbb N}$ and tolerance sequence $\{\epsilon_k\}_{k\in\mathbb N}$. Then  for all $k\in\mathbb N$, there holds that
	\begin{equation}\label{eq:rate}
\begin{aligned}
	{}&		\mathcal L(\bm u_{k},\lambda^*)-\mathcal L(\bm u^*,\lambda_k) +	
	\nm{H\bm u_k-b}+	\snm{h(x_k)-h(x^*)}\\
	\leq {}&
	\left( C_1(\sqrt{\widehat{\varepsilon}_k})+ C_2(\widetilde{\varepsilon}_k)\right)
	\times \prod_{i=0}^{k-1}\frac{1}{1+\alpha_i},
\end{aligned}
	\end{equation}
	where both $C_1(\cdot)$ and $C_2(\cdot)$ are quadratic functions.
\end{thm}
\begin{proof}
	Based on \cref{eq:bd-Ek1} and the proof of \cite[Lemma 3.3]{luo_accelerated_2021}, we are ready to establish 		
	\begin{equation}\label{eq:Ekbound2}
		\mathcal L(\bm u_{k},\lambda^*)-\mathcal L(\bm u^*,\lambda_k) \leq 			\mathcal{E}_k\leq 
		\beta_k\left(\sqrt{\mathcal{E}_0+Z\widehat{\varepsilon}_k}+\sqrt{2}Q \widetilde{\varepsilon}_k\right)^2,
	\end{equation}
	where $	Z := {}\nm{H}\left(\sigma\nm{x^*}+\nm{H^\top\lambda^*-\widetilde{c}}\right)$ and $	Q: = {} 1+(1+\sqrt{\sigma})\nm{H}$.
	Since $\bm u_k\in\Sigma$, we have
	\[
	0\leq h(x_k)-h(x^*)+\dual{\lambda^*,H\bm u_k-b} = \mathcal L(\bm u_k,\lambda^*)-\mathcal L(\bm u^*,\lambda_k)\leq \mathcal E_k\leq \beta_k R_k^2,
	\]
	where $R_k: = \sqrt{\mathcal{E}_0+Z\widehat{\varepsilon}_k}+\sqrt{2}Q \widetilde{\varepsilon}_k$, and it follows immediately that
	\begin{equation}
		\label{eq:est-hxk}
		{}	\snm{h(x_k)-h(x^*)}\leq \beta_kR_k^2+\nm{\lambda^*}		\nm{H\bm u_k-b},
	\end{equation}
	
	Below, we aim to prove
	\begin{equation}\label{eq:est-Huk-b}
		{}	\nm{H\bm u_k-b}\leq \beta_k\Big(\nm{H\bm u_0-b}+\nm{\lambda_0+\lambda^*}
		+\sqrt{2}R_k+\nm{H}^2\widehat{\varepsilon}_k+\widetilde{\varepsilon}_k\Big),
	\end{equation}
	which together with \cref{eq:Ekbound2,eq:est-hxk} proves \cref{eq:rate}. Note that $(\bm u_{k+1}^{\#},\bm v_{k+1}^{\#},\lambda_{k+1}^{\#})$ is the exact solution to the implicit Euler discretization \cref{eq:apd-im-x-im-l} at the $k$-th iteration and by \eqref{eq:apd-im-x-im-l-l} we have
	\[
	\lambda_{k+1}^{\#} - \lambda_k = \alpha_k/\beta_k(H\bm v_{k+1}^{\#}-b)= \beta_{k+1}^{-1}(H\bm u_{k+1}^{\#}-b) -\beta_{k}^{-1}\left(H\bm u_{k}-b\right).
	\]
	Therefore, a rearrangement gives
	\[
	\lambda_{k+1} - \lambda_k =  E_k+
	\beta_{k+1}^{-1}\left(H\bm u_{k+1}-b\right) -\beta_{k}^{-1}\left(H\bm u_{k}-b\right),
	\]
	where $E_k:=\lambda_{k+1}-\lambda_{k+1}^{\#} +
	\beta_{k+1}^{-1}H(\bm u_{k+1}^{\#}-\bm u_{k+1})$. This also leads to
	\[
	\lambda_{k}-\lambda_0=\beta_{k}^{-1}\left(H\bm u_{k}-b\right)-(H\bm u_{0}-b)+\sum_{i=0}^{k-1}E_i,
	\]
	and we get
	\[
	\nm{H\bm u_k-b}\leq \beta_k\left(\nm{H\bm u_0-b}+\nm{\lambda_k-\lambda_0}+ \sum_{i=0}^{k-1}\nm{E_i}\right).
	\]
	Invoking \cref{eq:uk1-est} implies
	\[
	\nm{E_k}\leq \|\lambda_{k+1}-\lambda_{k+1}^{\#} \|
	+\beta_{k+1}^{-1}\nm{H}\|\bm u_{k+1}^{\#}-\bm u_{k+1}\|
	\leq \epsilon_k\big(1+\alpha_k^2\beta_k^{-2}\nm{H}^2\big),
	\]
	and using the estimate \cref{eq:Ekbound2} promises that $\nm{\lambda_k-\lambda^*}\leq \sqrt{2}R_k$. Consequently, we obtain \cref{eq:est-Huk-b} and finish the proof of this theorem.
\end{proof}

According to \cref{eq:rate}, the final rate is obtained as long as the step size $\alpha_k$ and the error $\epsilon_k$ are specified. Two examples are given in order.
\begin{rem}\label{rem:ak-linear}
	Consider non-vanishing step size $\alpha_k\geq \widehat{\alpha}>0$. If $\epsilon_k = \mathcal O(\beta_k^{3/2}/(k+1)^{p})$ with $p>1$, then $\widetilde{\varepsilon}_k<\infty$ and $\sqrt{\beta_k}\widehat{\varepsilon}_k<\infty$. By \cref{eq:rate} and the fact that both $C_1(\cdot)$ and $C_2(\cdot)$ are quadratic functions, we obtain the final rate 
	\[
	\beta_k	\left( C_1(\sqrt{\widehat{\varepsilon}_k})+ C_2(\widetilde{\varepsilon}_k)\right) \leq C_3(\alpha_{k}^{\max})\sqrt{\beta_k},
	\]
	where $\alpha_{k}^{\max} = \max_{0\leq i\leq k-1}\{\alpha_i\}$ and $C_3(\cdot)$ is a quartic function. Therefore, we have at least linear rate since $\beta_{k}\leq (1+\widehat{\alpha})^{-k}$, and superlinear convergence follows  provided that $\alpha_k\to\infty$.
\end{rem}
\begin{rem}\label{rem:ak-sublinear}
	We then consider vanishing step size $\alpha_k\to 0$. In particular, assume $\alpha_{k}^2 = (k+1)^p\beta_k^3\beta_{k+1}^{-2}$ with $p>0$, then an elementary calculation yields that $\beta_k = \mathcal O(1/(k+1)^{2+p})$ and $\alpha_k = \mathcal O(1/(k+1))$. Hence, if $\epsilon_k = \mathcal O(1/(k+1)^q)$ with $q>3+2p$, then $\widehat{\varepsilon}_k+\widetilde{\varepsilon}_k<\infty$ and we have the sublinear rate
	\[
	\beta_k	\left( C_1(\sqrt{\widehat{\varepsilon}_k})+ C_2(\widetilde{\varepsilon}_k)\right) = \mathcal O(1/(k+1)^{2+p}),\text{ with any } p>0.
	\]
\end{rem}
\subsection{An SsN method for the subproblem \cref{eq:prox-lk1}}
\label{sec:SsN}
For $\delta_{\Sigma}(\bm u) = \delta_{\mathcal X}(x)+\delta_{\mathcal Y}(y)+\delta_{\mathcal Z}(z)$, define its Moreau--Yosida approximation
\begin{equation}\label{eq:def-MY}
	[\delta_\Sigma](\bm u): = \min_{\bm v\in\Sigma}\,
	\frac{1}{2}\nm{\bm v-\bm u}_{D_k}^2\quad\forall\,\bm u\in\R^{mn}\times\R^n\times \R^m,
\end{equation}
and introduce $	\mathcal F_k:\R^{m+n+r}\to\R$ by that 
\begin{equation}\label{eq:cal-Fk}
	\mathcal F_k(\lambda) := {}
	\frac{\beta_{k+1}}{2}\nm{\lambda}^2-\langle \widetilde{\lambda}_k,\lambda\rangle
	+\frac{1}{2}\nm{\bm w_k-H^\top\lambda}_{D_k^{-1}}^2-[\delta_{\Sigma}]\left(D_{k}^{-1}(\bm w_k-
	H^\top\lambda)\right).
\end{equation}
Note that $\mathcal F_k$ is strongly convex and continuous differentiable with $\nabla \mathcal F_k = F_k$, where $F_k$ has been defined in \cref{eq:Fk}. Indeed, according to \cite[Proposition 12.29]{Bauschke2011}, $[\delta_\Sigma](\cdot)$ is continuous differentiable and $	\nabla[\delta_\Sigma](\bm u) = D_k(\bm u-\proj_{\Sigma}(\bm u))$.
Moreover, by Moreau's decomposition \cite[Theorem 14.3 (ii)]{Bauschke2011}
\[
\bm u=\proj_{\Sigma}(\bm u) + D_k^{-1}\prox_{D_k\delta_\Sigma^*}(D_k\bm u) ,
\]
we also find that 
\begin{equation}\label{eq:cal-Fk-new}
\begin{aligned}
		\mathcal F_k(\lambda) = {}&
	\frac{\beta_{k+1}}{2}\nm{\lambda}^2-\langle \widetilde{\lambda}_k,\lambda\rangle
	+	\delta_{\Sigma}^*\left(\prox_{D_k\delta_{\Sigma}^*}(\bm w_k-H^\top\lambda)\right)\\
	{}&\quad	+\frac{1}{2}\nm{\proj_{\Sigma}(D_{k}^{-1}(\bm w_k-
		H^\top\lambda))}_{D_k}^2,
\end{aligned}
\end{equation}
where $\delta_{\Sigma}^* $ is the conjugate function of $\delta_{\Sigma}$ and $D_k\delta_\Sigma^*$ is understood as $\eta_k\delta_{\mathcal X}^*+\tau_k(\delta_{\mathcal Y}^*+\delta_{\mathcal Z}^*)$.

Let $M = m+n$.
As $\Sigma=\{\bm u\in\R^{mn+M}:\sigma_{1,i}\leq \bm u_i\leq \sigma_{2,i}\}$ is a box region, $\proj_{\Sigma}$ is piecewise affine and strongly semismooth (cf. \cite[Propositions 4.1.4 and 7.4.7]{Facchinei2003-v2}), and so is $F_k$ (see \cite[Proposition 7.4.4]{Facchinei2003-v2}). 
Denote by $\partial \proj_{\Sigma}(\bm u) $ the Clarke subdifferential \cite[Definition 2.6.1]{clarke_optimization_1987} of the proximal mapping $\proj_{\Sigma} $ at $\bm u$. Thanks to \cite[Table 3]{lin_augmented_2022}, we have
\begin{equation}\label{eq:par-proj}
	\partial \proj_{\Sigma}(\bm u) :=\left\{
	{\rm diag}(\chi):
	\chi_i\in
	\left\{
	\begin{aligned}
		&\{1\}&&\text{if}\,\sigma_{1,i}< \bm u_i< \sigma_{2,i}\\
		&[0,1]&&\text{if }\bm u_i \in\{\sigma_{1,i},\,\sigma_{2,i}\}\text{ and }
		\sigma_{1,i}\neq\sigma_{2,i}\\
		&\{0\}&&\text{if }\bm u_i \leq\sigma_{1,i}\text{ or }\bm u_i \geq\sigma_{2,i}
	\end{aligned}
	\right.
	\right\}.
\end{equation}

For every $\lambda\in\R^{M+r}$, let $U_k(\lambda) \in \partial \proj_{\Sigma}\left(D^{-1}_{k}(\bm w_k-H^\top\lambda) \right) $ and define an SPD matrix
\begin{equation}\label{eq:Hk}
	\mathcal J_k(\lambda) := \beta_{k+1}I+ HD_k^{-1}U_k(\lambda)H^{\top}.
\end{equation}
Then the semi-smooth Newton (SsN) iteration for \cref{eq:prox-lk1} reads as follows
\begin{subnumcases}{	\label{eq:SsN}}
	\label{eq:SsN-dj}
	\mathcal J_k(\lambda^j)\xi= -F_k(\lambda^j),\\
	\label{eq:SsN-lj}	
	\lambda^{j+1} = \lambda^j +\xi.
\end{subnumcases}
Since $F_k$ is strongly semismooth, we have local quadratic convergence \cite{Qi1993,Qi1993a}. Below, the SsN method \cref{eq:SsN} is summarized in \cref{algo:SsN}, where a line search procedure \cite{dennis_numerical_1996} is supplemented for global convergence.
\begin{algorithm}[H]
	\caption{SsN method for \cref{eq:prox-lk1}}
	\label{algo:SsN}
	\begin{algorithmic}[1] 
		\REQUIRE $\tau\in(0,1/2),\,\delta\in(0,1)$ and $\lambda\in\R^{M+r}$.
		\FOR{$j=0,1,\ldots$}
		\STATE Set $\lambda_{\rm old} = \lambda$ and $\bm z_k= D^{-1}_{k}(\bm w_k-H^\top\lambda)$.
		\STATE Compute $U_k(\lambda) \in \partial \proj_{\Sigma}\left(\bm z_k\right)$ by \cref{eq:par-proj}.
		\STATE Solve the linear SPD system $	\mathcal J_k(\lambda)\xi= -F_k(\lambda)$.
		\STATE Find the smallest $\ell\in\mathbb N$ such that $					\mathcal F_k(\lambda_{\rm old}+\delta^\ell \xi)\leq \mathcal F_k(\lambda_{\rm old})+\tau\delta^\ell\dual{F_k(\lambda_{\rm old}),\xi}$.
		\STATE Update $\lambda = \lambda_{\rm old} + \delta^\ell \xi$.
		\ENDFOR
	\end{algorithmic}
\end{algorithm}
To update the SsN iteration, we have to solve a linear SPD system in \eqref{eq:SsN-dj}. In \cref{sec:near-singu-SPD}, we shall explore its hidden graph structure and obtain an equivalent graph Laplacian, for which an efficient and robust algebraic multigrid method will be proposed in \cref{sec:amg}. 

\section{An Equivalent Graph Laplacian System}
\label{sec:near-singu-SPD}
\subsection{The reduced problem}
Recall that $H = (G,I_Y,I_Z)$, where $G,\,I_Y$ and $I_Z$ are defined in \cref{eq:A-b}. 
Let us rewrite \eqref{eq:SsN-dj}  in a generic form
\begin{equation}\label{eq:Ge-b}
	\mathcal H \xi = (\epsilon I+	\mathcal H_0)\xi=z,
\end{equation} 
where 
\begin{equation}\label{eq:calH-A}
	\mathcal H_0
	=\begin{pmatrix}
		\diag{t}+T \diag{s} T^\top&T \diag{s} \Pi^\top\\
		\Pi  \diag{s}  T^\top&\Pi  \diag{s} \Pi^\top
	\end{pmatrix},
\end{equation}
with $s\in\R_+^{mn}$ and $t\in\R_+^{M}$. Let $S = \diag{s}$ and $K =\diag{t}$ and 
write $\xi = (\xi_1,\xi_2)$ and $z = (z_1,z_2)$, then \cref{eq:Ge-b} is equivalent to 
\[
\left\{
\begin{aligned}
	(\epsilon I+		K+TS T^\top)\xi_1+TS \Pi^\top \xi_2 = z_1,\\
	\Pi S  T^\top \xi_1+(\epsilon I+\Pi S \Pi^\top)\xi_2 = z_2.
\end{aligned}
\right.
\]
Additionally, this gives
\[
\left\{
\begin{aligned}
	\xi_1={}&\left(\mathcal T-\Psi\widetilde{\Pi}^{-1}\Psi^\top\right)^{-1} \left(z_1-\Psi\widetilde{\Pi}^{-1}z_2\right),\\
	\xi_2 ={}& \widetilde{\Pi}^{-1}\left(z_2- \Psi^\top \xi_1\right),
\end{aligned}
\right.
\]
where $\mathcal T =\epsilon I+		K+TS T^\top\in\R^{M\times M},\,\widetilde{\Pi}=\epsilon I+\Pi S \Pi^\top\in\R^{r\times r}$ and $\Psi=TS \Pi^\top\in\R^{M\times r}$. 

Assume $r$ is small, then $\widetilde{\Pi}$ is easy to invert. This is true for all transport-like problems listed in \cref{sec:prob-exam}. Indeed, for partial optimal transport \cref{eq:PDOT}, $\widetilde{\Pi}$ is a constant ($r=1$) and for other problems, $\widetilde{\Pi}$ is just a vacuum ($r=0$). 
Moreover, thanks to Sherman--Woodbury formula, we have
\[
\left(\mathcal T-\Psi\widetilde{\Pi}^{-1}\Psi^\top\right)^{-1}=
\mathcal T^{-1}+\mathcal T^{-1}\Psi\left( \widetilde{\Pi}-\Psi^\top \mathcal T^{-1} \Psi \right)^{-1}\Psi^\top\mathcal T^{-1}.
\]
Since $\widetilde{\Pi}-\Psi^\top \mathcal T^{-1} \Psi \in\R^{r\times r}$ is invertible with small size, what we shall pay attention to is the inverse of $	\mathcal T$, which corresponds to the reduced linear system
\begin{equation}\label{eq:He-b}
	\mathcal T\xi=(\epsilon I+	K+TS T^\top)\xi = z.
\end{equation}
\subsection{An equivalent graph Laplacian}
Let $Y\in\R^{m\times n}$ be such that $\vect{Y} = s$, then a direct computation yields 
\begin{equation}\label{eq:cal-H}
	\mathcal T_0:=TS T^\top=
	\begin{pmatrix}
		\diag{Y^{\top}{\bf 1}_m}&	Y^{\top}\\
		&\\
		Y	&\diag{Y{\bf 1}_n}
	\end{pmatrix}.
\end{equation}
Besides, set $\mathcal Q= \diag{I_n,-I_m}$ and define $\mathcal A_0 := \mathcal Q\mathcal T_0\mathcal Q\in\R^{M\times M}$, then $\mathcal T_0 $ is spectrally equivalent to $\mathcal A_0$ and a  direct calculation gives 
\[
\mathcal A_0=
\begin{pmatrix}
	{\rm diag}(Y^{\top}{\bf 1}_m)&-Y^{\top}\\
	-Y&	{\rm diag}(Y{\bf 1}_n)
\end{pmatrix}.
\]
Note that $\mathcal A_0$ is the Laplacian matrix of the bipartite graph $\mathcal G = (\mathcal V,\mathcal E,w)$, where 
$w = \vect{Y},\,\mathcal V = \mathcal V_1\cup \mathcal V_2$ with $\mathcal V_1 = \{1,2,\cdots,n\}$ and $\mathcal V_2 = m+\mathcal V_1$, and $\mathcal E = \{e=\{i,j\}:w_{(i-1)m+j-n}>0,\,i\in \mathcal V_1,\,j\in \mathcal V_2\}$.
Consequently the reduced linear system \cref{eq:He-b} now is equivalent to 
\begin{equation}\label{eq:equi-calA-eps}
	\mathcal A u=(\epsilon I +K+\mathcal A_0) u= \mathcal Qz,
\end{equation}
where $K =\diag{t}$ is diagonal with nonnegative components $t\in\R_+^M$. Clearly, if $u$ solves \cref{eq:equi-calA-eps}, then the solution to \cref{eq:He-b} is given by $\xi = \mathcal Qu$. 
\begin{rem}	
	We claim that the sparsity pattern of $s$ (and thus $Y$) is related to that of $\{x_k\}_{k\in\mathbb N}$. Recall that 
	\[
\begin{aligned}
		\bm u_{k+1} = {}&\proj_{\Sigma}\left(D_{k}^{-1}(\bm w_k-
	H^\top\lambda_{k+1})\right),\\
	x_{k+1} ={}&  \proj_{\mathcal X}\left(\eta_{k}^{-1}(w_k-
	G^\top\lambda_{k+1})\right),
\end{aligned}
	\]
where $w_k$ is the component of $\bm w_k$ in $\mathcal X$. In view of \cref{eq:Hk}, we have 
\[
s \in\partial \proj_{\mathcal X}\left(\eta_{k}^{-1}(w_k-
G^\top\lambda_{k+1})\right),
\]
and by \eqref{eq:uk1} and \cref{eq:par-proj}, we see that $s$ is very close to the sparsity pattern of $x_{k+1}$. Moreover, as $x_k$ converges to $x^*$ that corresponds to an optimal transport plan $X^*$, the sparsity pattern of $Y$ agrees with that of $X^*$. 	
\end{rem}
\subsection{A hybrid framework}
We now discuss how to solve the linear SPD system \cref{eq:equi-calA-eps}.
If the bipartite graph $\mathcal G$ of $\mathcal A_0$ has $\kappa$ connected components, then there is a permutation matrix $\mathcal P$ such that
\begin{equation}\label{eq:diag-cal-A}
	\mathcal P^\top\mathcal A_0\mathcal P= \diag{	A_0^1, A_0^2,\cdots, A_0^\kappa},
\end{equation}
where each $A_0^i(1\leq i\leq \kappa)$ corresponds to the Laplacian matrix of some connected bipartite graph. Since $\epsilon I+K$ is diagonal, we are allowed to solve $\kappa$ independent linear systems, each of which takes the form
\begin{equation}\label{eq:Auf}
	A u=(\epsilon I+\Lambda+A_0) u= f,
\end{equation}
where $\Lambda$ is diagonal with nonnegative components and $A_0\in\R^{N\times N}$ is a connected graph Laplacian, with explicit null space: ${\rm span}\{{\bf 1}_N\}$. Note that if the diagonal part of $A_0$ has zero component, then it can be further reduced. Thus, without lose of generality, in what follows, assume all diagonal elements of $A_0$ are positive, which means $A_0$ has no zero row or column since $A_0{\bf 1}_N = 0$.

Recall that the size of the linear system \cref{eq:equi-calA-eps} is $M$-by-$M$. Therefore, if $N\leq M^{1/3}$, then the solution to \cref{eq:Auf} can be obtained via direct method within $\mathcal O(N^3)\leq  \mathcal O(M)$ complexity. Otherwise, we shall consider iterative methods. This leads to a hybrid approach, as summarized in \cref{algo:Hybrid_Solver}.
\begin{algorithm}[H]
	\caption{A Hybrid Solver for \cref{eq:He-b}: $\mathcal T\xi = z$}
	\label{algo:Hybrid_Solver}
	\begin{algorithmic}[1] 
		\STATE Set $\mathcal Q= \diag{I_n,-I_m},\,\mathcal A_0 = \mathcal Q\mathcal T_0\mathcal Q$ and $\mathcal A = \epsilon I+K+\mathcal A_0$.
		\STATE Check the connected components of $\mathcal A_0 $ and find a permutation matrix $\mathcal P$ such that \[
		\mathcal P^\top\mathcal A\mathcal P= \diag{ A_1, A_2,\cdots,A_\kappa}\quad\text{and}\quad 
		f = \mathcal P^\top \mathcal Qz= \left( f_1, f_2,\cdots,f_\kappa\right),
		\]
		where $A_i\in\R^{n_i\times n_i}$ and $f_i\in\R^{n_i}$, for all $1\leq i\leq \kappa$.
		\STATE For small component $n_i\leq M^{1/3}$, invoke direct method (or PCG) to solve $A_iu_i = f_i$.
		\STATE For large component $n_i>M^{1/3}$, apply iterative solver to $A_iu_i = f_i$.	
		\STATE Recover the solution $\xi = \mathcal Qu$ with $\left( u_1, u_2,\cdots,u_\kappa\right)$.
	\end{algorithmic}
\end{algorithm}

If $\Lambda\neq O$, then $A$ is SPD for all $\epsilon\geq 0$. When $\Lambda$ vanishes, $A$ becomes nearly singular if $\epsilon$ is close to zero. This tricky issue increases the number of iterations of standard solvers like Jacobi iteration, Gauss-Seidel iteration, and PCG; see our numerical evidence in \cref{tab:amg-pcg-test1-even}, and we refer to \cite{lee_robust_2007} for detailed discussions on this. Moreover, standard iterative methods are not robust concerning the problem size as well, which motivates us to consider the algebraic multigrid (AMG) algorithm.
\section{Classical AMG Algorithm}
\label{sec:amg}
Multigrid methods are efficient iterative solvers or preconditioners for large sparse linear SPD systems arising from numerical discretizations of partial differential equations (PDEs) \cite{bramble_parallel_1990,chen_convergence_2020,chen_optimal_2012,hackbusch_multi-grid_2011,li_bpx_2016,wu_convergence_2012,xu_theory_1989,xu_new_1992,xu_two-grid_1996}. Those linear systems are always  ill-conditioned as the mesh size decreases (or equivalently the problem size increases), and standard stationary iterative solvers converge dramatically slowly. However, multigrid methods possess {\it mesh-independent} convergence rate and can achieve the {\it optimal} complexity. 

The basic multigrid ingredients are error smoothing and coarse grid correction. In the setting of PDE discretizations, the coarse grid is based on geometric mesh and a multilevel hierarchy can also be constructed easily. On the other hand, the multigrid idea has been applied to the case where no geometric mesh is available. In particular, for the graph Laplacian system \cref{eq:Auf}, multilevel hierarchy can still be obtained from the adjacency graph to $A$. Then different coarsening techniques and interpolations lead to various algebraic multigrid algorithms, such as classical AMG and aggregation-based AMG \cite{hutchison_algebraic_2006,briggs_multigrid_2000,xu_algebraic_2017}.
\subsection{Multilevel $W$-cycle}
\label{sec:amg-w}
Let us first present an abstract multilevel $W$-cycle framework for solving \cref{eq:Auf}. There are two steps: the setup phase and the iteration phase. 

In the setup phase, we work with a family of coarse spaces: $ \{V_\ell = \R^{N_\ell}\}_{\ell = 1}^J$, where $N_{J}<\cdots<N_{\ell}<\cdots N_1 = N$, and build some basic ingredients that include
\begin{itemize}
	\item Smoothers: $R_\ell:\R^{N_{\ell}}\to \R^{N_\ell}$ for all $1\leq \ell\leq J$;
	\item Prolongation matrices: $P_\ell:\R^{N_{\ell+1}}\to \R^{N_\ell}$ for $1\leq \ell\leq J-1$;
	\item Coarse level operators: $A_1 = A$ and $A_{\ell+1} = P_{\ell}^{\top}A_{\ell}P_\ell$ for all $1\leq \ell\leq J-1$.
\end{itemize}
The coarsest level size $N_J$ is very small and in practice, $N_J = \mathcal O(N^{1/3})$  is acceptable. The prolongation operators $\{P_\ell\}_{\ell=1}^{J-1}$ shall be injective, i.e., each $P_{\ell}\in\R^{N_\ell\times N_{\ell+1}}$ has full column rank. Moreover, since $A$ might be nearly singular, we require that $P_{\ell}{\bf 1}_{N_\ell+1} = {\bf 1}_{N_{\ell}}$, then $A_\ell{\bf 1}_{N_\ell}$ is close to zero for all $1\leq \ell\leq J$. 

In each level, the smoother $R_{\ell}\in \R^{N_\ell\times N_\ell}$ is an approximation to $A^{-1}_{\ell}$ and possesses smoothing property. For $\ell = J$, we can choose $R_\ell = A_\ell^{-1}$ or invoke the PCG iteration. For $1\leq \ell<J$, let $A_\ell= D_\ell+L_\ell+L^{\top}_\ell$ where $D_\ell$ is the diagonal part and $L_\ell$ is the strictly lower triangular part. Then we can consider
\begin{itemize}
	\item Gauss--Seidel: $R_\ell=(D_{\ell}+L_{\ell})^{-1}$ for $\ell = 1$;
	\item Weighted Jacobi: $R_\ell = \omega D^{-1}_{\ell}$ with $\omega\in(0,1)$ for $1\leq \ell\leq J$.
\end{itemize}
For $\ell=1$, thanks to the bipartite graph structure of $A_1=A$, the Gauss--Seidel smoother admits explicit expression. Given a smoother $R_\ell$, to handle the possibly nearly singular property of $A_\ell$, we follow \cite{lee_robust_2007,padiy_generalized_2001} and adopt a special one
\begin{equation}\label{eq:Ra}
	\widehat{R}_{\ell} =\frac{\xi_\ell\xi_\ell^\top}{\xi_\ell^\top A_\ell\xi_\ell}+R_\ell\left(I-A_\ell\frac{\xi_\ell\xi_\ell^\top}{\xi_\ell^\top A_\ell\xi_\ell}\right),
\end{equation}
where $\xi_\ell = {\bf 1}_{N_\ell}$ is the approximation kernel of $A_\ell$. 

Then in the iteration phase, we run the process
\begin{equation}\label{eq:AMG-W-iter}
	u_{k+1} = u_k + \mathtt{AMG}{\textrm -}\mathtt{W}(f-Au_k,0,1),\quad k = 0,1,\cdots,
\end{equation}
where $g = \mathtt{AMG}{\textrm -}\mathtt{W}(\zeta,e,\ell)$ is defined by \cref{algo:MG-W} in a recursive way. 
\begin{algorithm}[H]
	\caption{Algebraic Multigrid $W$-cycle: $g = \mathtt{AMG}{\textrm -}\mathtt{W}(\zeta,e,\ell)$}
	\label{algo:MG-W}
	\begin{algorithmic}[1] 
		\REQUIRE $\zeta,\,e\in \R^{N_\ell},\,1\leq \ell\leq J$ and $\theta\in\mathbb N_{\geq 1}$.
		\IF{$\ell =J$}
		\STATE $g = e + \widehat{R}_\ell  (\zeta-A_{\ell}e)$.
		\ELSE
		\FOR[Presmoothing]{$i=1,2,\cdots,\theta$}
		\STATE $e = e + \widehat{R}_\ell  (\zeta-A_{\ell}e)$.
		\ENDFOR
		\STATE Restriction: $\zeta_{\ell+1} = P^\top_{\ell}(\zeta-A_{\ell}e)$.
		\STATE Coarse correction: $e_{\ell+1}= \mathtt{AMG}{\textrm -}\mathtt{W}(\zeta_{\ell+1},0,\ell+1)$.		
		\STATE Coarse correction: $e_{\ell+1}= \mathtt{AMG}{\textrm -}\mathtt{W}(\zeta_{\ell+1},e_{\ell+1},\ell+1)$.		
		\STATE Prolongation: $e = e+P_{\ell}e_{\ell+1}$.
		\FOR[Postmoothing]{$i=1,2,\cdots,\theta$}
		\STATE $e = e + \widehat{R}^\top_\ell (\zeta-A_{\ell}e)$
		\ENDFOR
		\ENDIF
	\end{algorithmic}
\end{algorithm}
\begin{rem}\label{rem:level-J}
	We mention that the number of smoothing iterations $\theta\in\mathbb N_{\geq 1}$ in \cref{algo:MG-W} is fixed and in most cases a small choice, saying $\theta=5$, works well. In addition, the coarsening procedure, which will be introduced in the next section, leads to the reduction $N_{\ell+1} \approx N_\ell/2$, and thus the total number of levels is at most $J = \mathcal O(\ln N)$. Consequently, if 
	\begin{itemize}
		\item[(i)] the convergence rate $\rho^k$ of the AMG $W$-cycle \cref{eq:AMG-W-iter} is robust, which means $\rho\in(0,1)$ is independent of the singular parameter $\epsilon$ in $A$ and the 
		problem size $N$, and 
		\item[(ii)] the matrix-vector operations in each iteration of \cref{eq:AMG-W-iter} is $\mathcal O(\mathtt{nnz}(A))$, where $\mathtt{nnz}(A) $ denotes the number of nonzero elements of $A$,
	\end{itemize}
	then to achieve a given tolerance $\varepsilon$, the total computational work of the AMG $W$-cycle \cref{eq:AMG-W-iter} is optimal $\mathcal O(\mathtt{nnz}(A)|\ln\varepsilon|)$
	
	The convergence rate of the two-level case will be established later in \cref{sec:conv-2gd}, and the efficiency of the multilevel $W$-cycle shall be verified by numerical tests in \cref{sec:num-part1-amg}.
\end{rem}
\subsection{Coarsening and interpolation}
In this part, we shall construct the prolongation operators $\{P_\ell\}_{\ell=1}^{J-1}$. In the terminology of AMG, it can be done by {\it coarsening} and {\it interpolation} \cite{trottenberg_multigrid_2001,xu_algebraic_2017}. Here, ``interpolation'' means the operator $P_\ell:\R^{N_{\ell+1}}\to\R^{N_\ell}$ provides a good approximation from the coarse level $\R^{N_{\ell+1}}$ to the fine level $\R^{N_\ell}$. According to the hierarchy structure, it is sufficient to consider the case $\ell = 1$, which provides a template for coarse levels $1<\ell\leq J$.
\subsubsection{Maximal independent set}
In classical AMG, the coarsening is based on the so-called $\mathcal C\backslash 
\mathcal F$-splitting. Recall that $A_1=A$ and $N_1=N$.
Let $\mathcal V = \{1,2,\cdots,N\}$ and define the {\it strength} function $s_{\!A}:\mathcal V\times \mathcal V\to\R$ with respect to $A$ by that
\begin{equation}\label{eq:sij}
	s_{\!A}(i,j): = \frac{A_{ij}}{\max\{\min_{k\in \mathcal N(i)}A_{ik},\,\min_{k\in \mathcal N(j)}A_{jk}\}}\quad\forall\,(i,j)\in \mathcal V\times \mathcal V,
\end{equation}
where $\mathcal N(i): = \left\{j\in\mathcal V\backslash\{i\}:A_{ij}\neq 0\right\}$.
Given a threshold $\delta\in(0,1)$, we say $i\in\mathcal V$ and $j\in\mathcal V$ are strongly connected if $s_{\!A}(i,j)>\delta$. 
We aim to find a {\it maximal independent set} $\mathcal C = \{j_1,j_2,\cdots,j_{N_{2}}\}\subset\mathcal V$, such that any $i\in\mathcal C$ and $j\in\mathcal C$ are not strongly connected, i.e., $s_{\!A}(i,j)\leq \delta$. Then $\mathcal C$ stands for the collection of coarse nodes and its complement $\mathcal F = \mathcal V\backslash\mathcal C=\{i_1,i_2,\cdots,i_{N_{\rm f}}\}$ denotes the set of fine nodes, where $N_{\rm f} = N - N_{2}$. Notice that for any $i\in\mathcal F$, $\mathcal C\cap\mathcal N(i)$ is nonempty.

A basic splitting algorithm (cf. \cite[Algorithm 5]{xu_algebraic_2017}) has been described briefly in \cref{algo:C-F}. We refer to \cite[Appendix A.7]{trottenberg_multigrid_2001} for an variant, where a measure of importance has been introduced to obtain a reasonable distribution of coarse nodes.
\begin{algorithm}[H]
	\caption{$\mathcal C\backslash 
		\mathcal F$-splitting}
	\label{algo:C-F}
	\begin{algorithmic}[1] 
		\STATE Set the threshold $\delta\in(0,1)$.
		\STATE Initialize $\mathcal C = \emptyset$ and $\mathcal F = \emptyset$.
		\STATE Mark all nodes in $\mathcal V$ as unvisited: $\mathcal U(i) = true$ for all $i\in\mathcal V$.
		\FOR{$i = 1,2,\cdots,N$}
		\IF[$i$ has not been visited]{$\mathcal U(i)=true$}
		\STATE $\mathcal N_{\rm s}(i)=\{j\in\mathcal V:s_{\!A}(i,j)>\delta\}$.
		\STATE $\mathcal C = \mathcal C\cup\{i\}$ and $\mathcal F = \mathcal F\cup\mathcal N_{\rm s}(i)$.
		\STATE $\mathcal U(i) =false$ and $\mathcal U(k) = false$ for all $k\in\mathcal N_{\rm s}(i)$.
		\ENDIF
		\ENDFOR
	\end{algorithmic}
\end{algorithm}
\subsubsection{Interpolation operator}
\label{sec:iterp}
Once the $\mathcal C\backslash\mathcal F$-splitting has been done, 
we can find a permutation matrix $\Xi$ such that
\begin{equation}\label{eq:Xi}
	\Xi^\top A\Xi = \begin{pmatrix}
		A_{FF}&A_{FC}\\
		A_{FC}^\top&A_{CC}
	\end{pmatrix}.
\end{equation}
Then, we can choose (see \cite[Section 12.3]{xu_algebraic_2017})
\begin{itemize}
	\item {\it Ideal interpolation}: $P= \Xi \begin{pmatrix}
		W\\ I
	\end{pmatrix}$ with $W  = 		-A_{FF}^{-1}A_{FC}$.
	\item {\it Standard interpolation}: $P=\Xi \begin{pmatrix}
		(I-D_{FF}^{-1}A_{FF})	W\\ I
	\end{pmatrix}$ with $D_{FF}$ being the diagonal part of $A_{FF}$.
\end{itemize}

In addition, to satisfy $P{\bf 1}_{N_{2}} = {\bf 1}_{N_1}$, we need a scaling transform $P_1= \diag{P{\bf 1}_{N_{2}}}\backslash P$, which leads to the desired prolongation operator from level $\ell = 2$ to level $\ell =1$. 

Observing the particular structure of the system \cref{eq:Auf}, where $A_0$ is the Laplacian of some connected bipartite graph, we find
\[
A = \begin{pmatrix}
	A_{FF}&A_{FC}\\
	A_{FC}^\top&A_{CC}
\end{pmatrix},
\]
with $A_{FF} $ and $A_{CC} $ being diagonal. This yields an approximate $ \mathcal C\backslash\mathcal F$-splitting
\[
\mathcal F = \{1,2,\cdots,n_{\rm f}\}\quad\text{and}\quad \mathcal C = n_{\rm f}+\{1,2,\cdots,n_{\rm c}\},
\]
where $n_{\rm f}+n_{\rm c} = N$. Note that $\mathcal C$ might not be a maximal independent set but provides an approximate ideal interpolation. 

However, for $\ell>1$, it is not realistic to expect the bipartite structure of $A_\ell$, and to avoid inverting $A_{FF}$, we shall consider standard interpolation instead.

\section{Convergence Analysis}
\label{sec:conv-2gd}
Given $\zeta_{\ell},\,e_\ell\in\R^{N_\ell}$ and $1\leq\ell\leq J$, let $g_\ell = (I-B_{\ell}A_\ell)e_\ell+B_\ell\zeta_{\ell}$ be the output of \cref{algo:MG-W}, then $B_{J}: = \widehat{R}_J$, and by induction, for $1\leq \ell<J$, we have the recurrence relation
\begin{equation}\label{eq:I-BlAl}
	I-B_{\ell}A_\ell = 
	\left(I-\widehat{R}_\ell^\top A_\ell\right)^{\theta}(I-P_\ell B_{\ell+1}^\top P_\ell^\top A_\ell)
	(I-P_\ell B_{\ell+1}P_\ell^\top A_\ell)
	\left(I-\widehat{R}_\ell A_\ell\right)^{\theta}.
\end{equation}
Correspondingly, one finds that \cref{eq:AMG-W-iter} becomes
\begin{equation}\label{eq:uk-B1}
	u_{k+1} =  u_k +B_1( f- A u_k),\quad k = 0,1,\cdots,
\end{equation}
which yields 
\[
\nm{u_{k+1}-u^*}_A\leq \nm{I-B_1A}_{\!A} \nm{u_{k}-u^*}_A,
\]
with $u^*=A^{-1}f$ being the exact solution.

In the following, we aim to establish the estimate of $\nm{I-B_1A}_{\!A}$, by using the well-known Xu--Zikatanov identity \cite{xu_method_2002}. The main result is summarized in \cref{thm:conv-2g-AMG}, which says that the convergence rate of the two level case is independent of the singular parameter $\epsilon$ and the size $N$.
\begin{thm}
	\label{thm:conv-2g-AMG}
	Assume that $R_1$ is the Gauss--Seidel smoother (cf. \cref{sec:amg-w}) and $P_1$ is the ideal interpolation (cf. \cref{sec:iterp}). If $\bar B_2: = B_2^\top+B_2-B_2^\top A_2B_2$ is SPD, then $B_1$ is SPD and
	\begin{equation}\label{eq:conv-2gd}
		\nm{I-B_1A}_{\!A} =  1-\frac{1}{c_1},\quad c_1\leq 1+C +\frac{1}{1-	\nm{I-\bar B_2A_2}_{A_2}},
	\end{equation}
	where $C>0$ is independent of $\epsilon$ and $N$.
\end{thm}
\begin{rem}
	Rigorously speaking, \cref{thm:conv-2g-AMG} does not provide final rate of the multilevel $W$-cycle \cref{eq:AMG-W-iter}, since the upper bound involves $\nm{I-\bar B_2A_2}_{A_2}$, which is related to the coarse level solver. The ideal case $\bar B_2 = A_2^{-1}$ implies $\nm{I-\bar B_2A_2}_{A_2}=0$, and thus $c_1\leq 2+C$. For the multilevel hierarchy, the estimate \cref{eq:conv-2gd} provides essential evidence to show that the AMG $W$-cycle can be robust provided that the coarse level solver  works well.
\end{rem}
\subsection{A robust estimate}
For ease of notation, we set $P= P_1$ and $N_{\rm c} = N_2$. Besides, let $V = \R^N$ and $\bm X = \R\times \R^{N}\times  \R^{N_{\rm c}}$, and define 
$	\Vc: = {\rm span}\left\{v_{\rm c} = Px:x\in\R^{N_{\rm c}}\right\}$ and $\Vn := {}{\rm span}\{\xi\}$ with $ \xi = {\bf 1}_N$.

Introduce $\bar B$ by that 
\begin{equation}\label{eq:barB}
	I-\bar BA = 
	(I-\widehat R_1^\top A)
	(I-PB_2^\top P^\top A)(I-PB_2P^\top A)(I-\widehat R_1A).
\end{equation}
Then by \cref{eq:I-BlAl}, we have 
\begin{equation}\label{eq:B1tobarB}
	I-B_1A = (I-\widehat{R}_1^\top A)^{\theta-1}\left(I-\bar BA\right)	(I-\widehat{R}_1 A)^{\theta-1}.
\end{equation}
The following lemma says that we only need to focus on $\nm{I-\bar BA}_A$, which corresponds to the simple case $\theta = 1$.
\begin{lem}
	\label{lem:RtoR1}
	If $\bar R=R_1^\top+R_1-R_1^\top AR_1$ is SPD, then 
	\[
	\|I-\widehat R_1A\|_A<1 \quad\text{and}\quad	\|I-\widehat R_1^\top A\|_A<1.
	\]
	Moreover, we have $\nm{I-B_1A}_{\!A} < \|I-\bar BA\|_{\!A} $.
\end{lem}
\begin{proof}
	According to the proof of \cite[Theorem 5, page 23]{xu_multilevel_2017}, we know that
	\begin{equation}\label{eq:barRA}
		\nm{I-\bar RA}_A=	\nm{I-R_1A}_A^2<1.
	\end{equation}
	Let $\widetilde R=R_1R_1^{-\top}\bar R  R_1^{-1}R_1^\top=R_1+R_1^\top-R_1AR_1^\top $. Note that $\widetilde{R}$ is SPD and 
	\[
	\|I-\widetilde RA\|_A=	\nm{I-R_1^\top A}_A^2<1.
	\]
	By \cref{eq:Ra}, we have $I-\widehat{R}_1A=(I-R_1A)(I-P_{\rm n})$, where $P_{\rm n} :V\to \Vn$ denotes the orthogonal projection operator with respect to the $A$-inner product, i.e., 
	\[
	P_{\rm n}= \xi(\xi^\top A\xi)^{-1}\xi^\top A = \frac{1}{\eta}\xi\xi^\top A,\quad \eta = \xi^\top A\xi>0.
	\]
	It follows from \cref{eq:barRA} that 
	\[
	\|I-\widehat R_1A\|_A
	\leq \nm{I-R_1A}_A\nm{I-P_{\rm n}}_A	\leq \nm{I-R_1A}_A<1.
	\]
	Similarly, we have $I-\widehat{R}^\top_1A=(I-P_{\rm n})(I-R_1^\top A)$ and thus $	\|I-\widehat R_1^\top A\|_A<1$.
	
	Hence, by \cref{eq:B1tobarB}, we obtain $\nm{I-B_1A}_{\!A} <\|I-\bar BA\|_{\!A} $ and conclude the proof.
\end{proof}

The notation $\bar R=R_1^\top+R_1-R_1^\top AR_1$ will be used in the sequel. For simplicity, for any $\bm x\in\bm X$, we write it as $\bm x = (x_1,x_2,x_3)$ with $x_1\in \R,\,x_2\in \R^{N}$ and $x_3\in \R^{N_{\rm c}}$, and $\varPi = (\xi,I,P)$ shall be understood as a linear mapping from $\bm X$ to $ V$ in the sense that $\varPi\bm x  := x_1\xi+x_2+Px_3\in V$.

One can consult many existing works \cite{huang_deriving_2011,xu_method_2002} on the proof of the following X-Z 
identity in more general abstract settings, and we refer to \cite[Theorem 22, page 64]{xu_multilevel_2017} for a comprehensive study 
\begin{lem}[X-Z identity]
	\label{lem:2g-XZ}
	Assume $P$ has full column rank and both $\bar R$ and $\bar B_2$ are SPD, then $\bar B$ is SPD and
	\[
	\nm{I-\bar BA}_{\!A} = 1-\frac{1}{c_0},\quad 	c_0 := \sup_{v\in V,\,\nm{v}_{\!A}=1}\inf_{\bm x\in\bm X,\,\varPi\bm x= v}\mathcal K(\bm x),
	\]
	where 
	\[
	\mathcal K(\bm x): = 
	\eta\left(	x_1+\eta^{-1}\xi^\top A(x_2+Px_3)\right)^2+\nm{x_2+R_1^{\top}APx_3}_{\bar R^{-1}}^2+\nm{x_3}_{\bar B_2^{-1}}^2.
	\]
\end{lem}

Define the $A$-orthogonal component of $\Vn$ as follows
\begin{equation}\label{eq:V0-perp}
	\Vr: = \left\{v_{\rm r}\in V:\dual{v_{\rm r},v_n }_{\!A}= 0\quad\forall\,v_n\in \Vn\right\}.
\end{equation}
It is evident that $V= \Vn\oplus \Vr$ and any $v\in V$ admits a unique $A$-orthogonal decomposition $v = v_{\rm n}+v_{\rm r}$, where $v_{\rm n}\in \Vn$ and $v_{\rm r}\in \Vr$ satisfy
\begin{equation}\label{eq:A_e-norm}
	\nm{v}_{\!A}^2 =  	\nm{v_{\rm n}}_{\!A}^2 +	\nm{v_{\rm r}}_{\!A}^2 +2\dual{v_{\rm r},v_{\rm n}}_{\!A} = \nm{v_{\rm n}}_{\!A}^2 +	\nm{v_{\rm r}}_{\!A}^2.
\end{equation}

Based on \cref{lem:2g-XZ}, we can establish a robust estimate of the constant $c_0$, which gets rid of the singular parameter $\epsilon$ in $A$. For the analysis of more general cases, we refer to \cite{lee_robust_2007}.
\begin{lem}\label{lem:chi1}
	Under the assumption of \cref{lem:2g-XZ}, we have $c_0\leq 1+\chi_1$ where
	\[
	\chi_1:= \sup_{v_{\rm r}\in \Vr,\,\nm{v_{\rm r}}_{\!A}=1}	\inf_{\substack{x_2+Px_3= v_{\rm r}\\x_2\in\R^N,x_3\in\R^{N_{\rm c}}}}
	\left\{	\nm{x_2+R_1^{\top}APx_3}_{\bar R^{-1}}^2+\nm{x_3}_{\bar B_2^{-1}}^2\right\}.
	\]
\end{lem}	
\begin{proof}
	Given any fixed $v\in V$, we have the unique decomposition $v=v_{\rm n}+v_{\rm r}$. Consider $\bm z=(z_1,0,0)$ with $z_1= v_{\rm n}^\top \xi/\nm{\xi}^2$, then $v_{\rm n} = \varPi \bm z\in V_{\rm n}$. It is clear that there exists at least one $\bm y = (0,y_2,y_3)\in\bm X$ such that $v_{\rm r} = \varPi\bm y\in V_{\rm r}$. Therefore, $\bm x = \bm y+\bm z$ is a special decomposition and $\varPi\bm x = v$. Observing the estimate
	\[
	\begin{aligned}
		{}&	\inf_{\bm x\in\bm X,\,\varPi\bm x= v}\left\{
		\eta\left(	x_1+\eta^{-1}\xi^\top A(x_2+Px_3)\right)^2+
		\nm{x_2+R^\top_1 APx_3}_{\bar R^{-1}}^2+\nm{x_3}_{\bar B_2^{-1}}^2\right\}\\
		\leq{}&	\inf_{\substack{\varPi\bm y= v_{\rm r}\\\bm y = (0,y_2,y_3)\in\bm X}}\left\{
		\eta\left(	z_1+\eta^{-1}\xi^\top A\varPi\bm y\right)^2+
		\nm{y_2+R^\top_1 APy_3}_{\bar R^{-1}}^2+\nm{y_3}_{\bar B_2^{-1}}^2\right\}\\	
		={}&\eta\snm{z_1}^2+	\inf_{\substack{y_2+Py_3= v_{\rm r}\\y_2\in\R^N,y_3\in\R^{N_{\rm c}}}}
		\left\{
		\nm{y_2+R^\top_1 APy_3}_{\bar R^{-1}}^2+\nm{y_3}_{\bar B_2^{-1}}^2\right\},
	\end{aligned}
	\]
	we obtain from \cref{lem:2g-XZ} and the fact $	\eta\snm{z_1}^2= \nm{v_{\rm n}}_{\!A}^2$ that $c_0\leq 1+\chi_1$.
\end{proof}
\subsection{Proof of \cref{thm:conv-2g-AMG}}
To move on, we prepare two key lemmas that are crucial for the proof of \cref{thm:conv-2g-AMG}. One is the {\it smoothing property} (see \cref{lem:smooth}) and the other is the {\it approximation property} (see \cref{lem:approx}). 
\begin{lem}[Smoothing property]
	\label{lem:smooth}
	If $R_1$ is the Gauss--Seidel smoother, then $\bar R$ is SPD, and $\nm{R^\top_1 Av}_{\bar R^{-1}}\leq C_s\nm{v}_A$ for all $v\in V$, where $C_s>0$ is independent of $\epsilon$ and $N$. 
\end{lem}
\begin{proof}
	Recall the splitting $A= D+L+L^{\top}$ where $D$ is the diagonal part and $L$ is the strictly lower triangular part. Since $R_1 = (D+L)^{-1}$, it follows that $\bar{ R}= R_1^\top(R_1^{-1}+R_1^{-\top}-A)R_1 = R_1^\top DR_1$ is SPD. 	
	
	Then let us prove that $	\nm{R^\top_1 Av}_{\bar R^{-1}}\leq C_s\nm{v}_A$ for all $v\in V$. Thanks to \cite[Lemma 3.3]{zikatanov_two-sided_2008}, we have 	
	\begin{equation}\label{eq:barR-D}
		\frac{1}{4}	\dual{Dv,v}\leq	\dual{\bar{R}^{-1}v,v}\leq C_{s}^2\dual{Dv,v},
	\end{equation}
	where $		C_s: = \max_{1\leq i\leq N}\snm{\mathcal N(i)}$. Since $\mathcal N(i): = \left\{j\neq i:A_{ij}\neq 0\right\}$, we note that $C_s$ depends only on the sparsity of $A$. In view of \cref{eq:barR-D} and the identity 
	\[
	\dual{\bar{R}^{-1}v,v}=\nm{L^\top v}_{D^{-1}}^{2} + \nm{v}^2_{\!A},
	\]
	we find that $\dual{D^{-1}v,v}\leq C_s^2\dual{A^{-1}v,v}$. Consequently, as $D^{-1}= R_1\bar{R}^{-1}R_1^\top$, we obtain
	\[
	\begin{aligned}
		\nm{R^\top_1 Av}_{\bar R^{-1}}^2  ={} \dual{D^{-1}Av,Av}\leq C_s^2\nm{v}^2_{\!A},
	\end{aligned}
	\]
	which ends the proof of this lemma.
\end{proof}

Let $\mathcal F = \{i_1,i_2,\cdots,i_{n_{\rm f}}\}$ and $\mathcal C = \{j_1,j_2,\cdots,j_{n_{\rm c}}\}$
be the $ \mathcal C\backslash\mathcal F$-splitting of $A$ with respect to the strength parameter $\delta>0$. Denote by
$e_i$ the $i$-th canonical basis of $ \R^N$ and set
\[
W_{\rm f}: = {\rm span}\{e_{k}:k\in\mathcal F\},\quad W_{\rm c}: = {\rm span}\{e_k:k\in \mathcal C\}.
\]
Let 
$P_{\rm f}:V\to W_{\rm f}$ be the $A$-orthogonal projection, i.e.,
\[
P_{\rm f} = \widetilde{W}_{\rm f}(\widetilde{W}_{\rm f}^\top A\widetilde{W}_{\rm f})^{-1}\widetilde{W}_{\rm f}^\top A,
\]
with $\widetilde{W}_{\rm f} = (e_{i_1},e_{i_2},\cdots,e_{i_{n_{\rm f}}})\in\R^{N\times n_{\rm f}}$. 
\begin{lem}[Approximation property]\label{lem:approx}
	Let $R_1$  be  the Gauss--Seidel smoother and $P_1$ is the ideal interpolation, then for any $v\in V$, we have $v_0=(I-P_{\rm f})v\in \Vc$ and 
	\begin{equation}\label{eq:key}
		\nm{v-v_0}_{ \bar R^{-1}}\leq C_s\sqrt{1+C_s/\delta}\nm{v}_{A},\quad \nm{v_0}_{\!A}\leq \nm{v}_{\!A}.
	\end{equation}
\end{lem}
\begin{proof}
	Recall that $	\Vc = {\rm span}\left\{v_{\rm c} = P_1x:x\in\R^{N_{\rm c}}\right\}$.
	By direct computations, we have $	P_1 = \widetilde{W}_{\rm c}-P_{\rm f}\widetilde{W}_{\rm c}$ where $\widetilde{W}_{\rm c} = (e_{j_1},e_{j_2},\cdots,e_{j_{n_{\rm c}}})\in\R^{N\times n_{\rm c}}$.  Thus it is evident that 
	\[
	\Vc={\rm span}\{v_{\rm c} = (I-P_{\rm f})w_{\rm c}:w_{\rm c}\in W_{\rm c}\}.
	\]
	On the other hand, since $V = W_{\rm c}\oplus W_{\rm f}$, any $v\in V$ admits a decomposition $v = w_{\rm c}+w_{\rm f}$ with $w_{\rm c}\in W_{\rm c}$ and $w_{\rm f}\in W_{\rm f}$. It follows from the fact $P_{\rm f}w_{\rm f} = w_{\rm f}$ that $(I-P_{\rm f})v = (I-P_{\rm f})w_{\rm c}\in V_{\rm c}$.
	
	Then let us focus on \cref{eq:key}.
	It is trivial to obtain $\nm{v_0}_{\!A}\leq \nm{v}_{\!A}$. 	Following the idea of \cite[Theorem 12.3]{xu_algebraic_2017}, we are already to establish
	\begin{equation}\label{eq:Cstab}
		\nm{v_{\rm f}}_{ D}^2\leqslant \left(1+C_s/\delta\right)\nm{v_{\rm f}}_{A}^2
		\quad\forall\,v_{\rm f}\in W_{\rm f}.
	\end{equation}
		In view of \cref{eq:barR-D}, it holds that
		\[
		\nm{v-v_0}_{ \bar R^{-1}}^2=\nm{P_{\rm f}v}_{\bar R^{-1}}^2
		\leq C_s^2\nm{P_{\rm f}v}_{D}^2
		\leq\left(C_s^2+C^3_s/\delta\right)\nm{v}_{A}^2,
		\]
		which finishes the proof of \cref{eq:key}.
	\end{proof}
	
	Borrowing the idea from \cite{li_two-level_2015,li_analysis_2016}, we are now in a position to prove \cref{thm:conv-2g-AMG}. 
	\vskip0.2cm\noindent{\bf Proof of \cref{thm:conv-2g-AMG}} 
	It is clear that the ideal interpolation $P_1$ has full rank. Since $R_1$ is the Gauss-Seidel smoother, by \cref{lem:smooth}, $\bar R$ is SPD. Hence, using \cref{lem:RtoR1,lem:2g-XZ,lem:chi1} leads to 
	\begin{equation}\label{eq:c1}
		\nm{I-B_1A}_A = 1-\frac{1}{c_1},\quad c_1\leq 1+\chi_1,
	\end{equation}
	where 
	\[
	\chi_1= \sup_{v_{\rm r}\in \Vr,\,\nm{v_{\rm r}}_{\!A}=1}	\inf_{\substack{x_2+P_1x_3= v_{\rm r}\\x_2\in\R^N,x_3\in\R^{N_{\rm c}}}}
	\left\{	\nm{x_2+R^{\top}_1AP_1x_3}_{\bar R^{-1}}^2+\nm{x_3}_{\bar B_2^{-1}}^2\right\}.
	\]
	Using \cref{lem:smooth} again, we have 
	\[
	\nm{x_2+R^{\top}_1AP_1x_3}_{\bar R^{-1}}^2\leq 
	2\nm{x_2}_{\bar R^{-1}}^2+2\nm{R^{\top}_1AP_1x_3}_{\bar R^{-1}}^2
	\leq 2\nm{x_2}_{\bar R^{-1}}^2+2C^2_s\nm{P_1x_3}_A^2.
	\]
	Since both $\bar B_2$ and $A_2 = P_1^\top AP_1$ are SPD, invoking the proof of \cite[Theorem 5, page 23]{xu_multilevel_2017}, we have $\nm{I-\bar B_2A_2}_{A_2} = 1-\lambda_{\min}(\bar B_2A_2)<1$, where $\lambda_{\min}(\bar B_2A_2)>0$ denotes the smallest eigenvalue of $\bar B_2A_2$. This yields that
	\[
	\nm{x_3}_{\bar B_2^{-1}}^2 \leqslant \frac{\nm{x_3}_{A_2}^2}{ \lambda_{\min}(\bar B_2A_2)}	 = \frac{\nm{P_1x_3}_{\!A}^2}{1-	\nm{I-\bar B_2A_2}_{A_2}}.
	\]
	Combining the above two estimates gives
	\[\small
	\begin{aligned}
		{}&
		\nm{x_2+R^{\top}_1AP_1x_3}_{\bar R^{-1}}^2+\nm{x_3}_{\bar B_2^{-1}}^2
		\leqslant {}
		2\nm{x_2}_{\bar R^{-1}}^2+
		\left(2C_s^2+\frac{1}{1-	\nm{I-\bar B_2A_2}_{A_2}}\right)	
		\nm{P_1x_3}_{\!A}^2.
	\end{aligned}
	\]
	Consequently, we arrive at 
	\[
	\chi_1\leq {}	\sup_{v_{\rm r}\in \Vr,\,\nm{v_{\rm r}}_{\!A}=1}	\inf_{v_{\rm c}\in \Vc}
	\left\{
	2\nm{v_{\rm r}-v_{\rm c}}_{\bar R^{-1}}^2+
	\left(2C_s^2+\frac{1}{1-	\nm{I-\bar B_2A_2}_{A_2}}\right)	
	\nm{v_{\rm c}}_{\!A}^2\right\}.
	\]
	Thanks to \cref{eq:barR-D,lem:approx}, taking $v_{\rm c} = (I-P_{\rm f})v_{\rm r}\in V_{\rm c}$ implies that 
	\[
	\chi_1\leq 2C_s^2\left(2+C_s/\delta\right)+
	\frac{1}{1-	\nm{I-\bar B_2A_2}_{A_2}}.
	\]
	Plugging this into \cref{eq:c1} leads to \cref{eq:conv-2gd} and thus concludes the proof of \cref{thm:conv-2g-AMG}. 
\section{Numerical Tests}
\label{sec:num}
This section is devoted to essential numerical experiments for validating the efficiency of our algorithm. In \cref{sec:num-part1-amg}, we aim to verify the robust performance of the AMG $W$-cycle iteration \cref{eq:AMG-W-iter} for solving a nearly singular graph Laplacian system. Then in \cref{sec:num-part2}, we apply the overall semismooth Newton-AMG-based inexact primal-dual  method (see \cref{algo:ipd-ssn-amg}) to several transport-like problems listed in \cref{sec:prob} and conduct extensive compassions with existing baseline algorithms. All numerical tests are implemented in MATLAB (version R2021a) on a MacBook Air Laptop.

\subsection{Performance of AMG}
\label{sec:num-part1-amg}
Consider the linear algebraic system
\begin{equation}\label{eq:amg-pcg-test1}
	Ax = \left(\epsilon I+A_h\right)x = f.
\end{equation}
Here $A_h$ is the stiffness matrix of a conforming bilinear finite element method \cite{brenner_mathematical_2008} for the pure Neumann problem 
\[\left\{
\begin{aligned}
	{}&	-\Delta u = g&&\textrm{in}~\Omega:=(0,1)^2,\\
	{}&\nabla u\cdot\mathbf{n} = 0&&\textrm{on}~\partial \Omega,
\end{aligned}
\right.
\]
where $\mathbf{n}$ is the unit outward normal vector of  $\partial \Omega$ and $g:\Omega\to \R$ is a square integrable function with vanishing average. 
\begin{figure}[H]
	\centering
	\includegraphics[width=10cm]{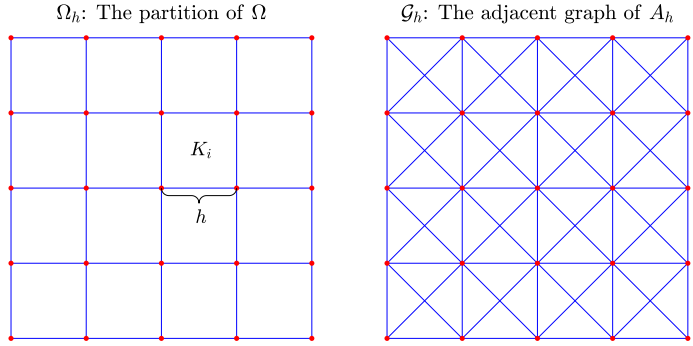}
	\caption{Illustrations of $\Omega_h$ and $\mathcal G_h$. }
	\label{fig:amg-pcg-test1-mesh}
\end{figure}

Let $\Omega_h = \mathop{\cup}_{i}K_i$ be a subdivision of $\Omega$, where each $K_i$ is a square with edge length $h = 2^{-k},\,k \in \mathbb N$. The  stiffness matrix $A_h\in\R^{N_h\times N_h}$ corresponds to $\Omega_h$ is sparse and symmetric positive semidefinite with $N_h = (1+1/h)^2$. Moreover, $A_h$ is the Laplacian matrix of some connected graph $\mathcal G_h$, which can be obtained from $\Omega_h$ by adding the two diagonal lines of each element $K_i$; see \cref{fig:amg-pcg-test1-mesh}.

We apply AMG $W$-cycle iteration \cref{eq:AMG-W-iter} and PCG (cf. \cite[Algorithm 9.1]{saad_iterative_2003}) to \cref{eq:amg-pcg-test1} with different $\epsilon$ and mesh size $h$. For PCG, we choose the diagonal (Jacobi) preconditioner. For AMG, we adopt weighted Jacobi smoother $R_\ell =  1/2 D_\ell$ and the number of smoothing iteration is $\theta = 5$. Additionally, to obtain a maximal independent set via \cref{algo:C-F} and avoid the for loop, we adopt a subroutine from the MATLAB software package: $i$FEM \cite{Chen:2008ifem}.
\renewcommand\arraystretch{1.3}
\begin{table}[H]
	\centering
	\caption{Number of iterations of AMG and PCG. }
	\label{tab:amg-pcg-test1-even}
	\setlength{\tabcolsep}{3pt}
	\begin{tabular}{|c|c|c|c|c|c|c|c|c|c|c|}
		\hline
		\multirow{2}{*}{$1/h$}&\multicolumn{2}{c|}{$\epsilon = 10^{-4}$}
		&\multicolumn{2}{c|}{$\epsilon = 10^{-6}$}
		&\multicolumn{2}{c|}{$\epsilon = 10^{-8}$}
		&\multicolumn{2}{c|}{$\epsilon = 10^{-10}$}
		&\multicolumn{2}{c|}{$\epsilon = 0$}\\
		\cline{2-11}
		
		&$\mathtt{itamg}$  &$\mathtt{itpcg}$ 
		&$\mathtt{itamg}$  &$\mathtt{itpcg}$ 
		&$\mathtt{itamg}$  &$\mathtt{itpcg}$ 		
		&$\mathtt{itamg}$  &$\mathtt{itpcg}$ 
		&$\mathtt{itamg}$  &$\mathtt{itpcg}$ 		\\ 
		\hline
		
		$2^{4} $                 & 9  & 62  & 10  & 66  & 9 & 69  & 9 & 52 & 10 & 52\\ \hline
		$2^{6} $                 & 9  & 230  & 9  & 250  & 9 & 266  & 9 & 201 & 10 & 201\\ \hline
		$2^{8} $                 & 9  & 789  & 9  & 906  & 9 & 969  & 10 & 740 & 9 & 740\\ \hline
		$2^{10} $               & 9  & 1427  & 10  & 3158  & 9 & 3531  & 10 & 2680 & 9 & 2680\\ \hline
	\end{tabular}
\end{table}

In \cref{tab:amg-pcg-test1-even}, we report the number of iterations of AMG (cf. $\mathtt{itamg}$) and PCG (cf. $\mathtt{itpcg}$), under the stop criterion
\[
\frac{\nm{Ax_k-f}}{\nm{Ax_0-f}}\leqslant \mathtt{Tol}=10^{-11}.
\]
As we can see, AMG is very robust with respect to both the singular parameter $\epsilon$ and the problem size $N_h$. While the number of PCG iterations grows in terms of $N_h$. If $\epsilon $ is decreasing and larger than $\mathtt{Tol}$, then due to the nearly singular issue, $\mathtt{itpcg}$ also increases. When $\epsilon$ is close to (or is smaller than) $\mathtt{Tol}$, the term $\epsilon I$ in $A$ is negligible and \cref{eq:amg-pcg-test1} can be viewed almost as a singular system. In this situation, $\mathtt{itpcg}$ tends to the case $\epsilon = 0$. To further show this dependence on $\epsilon$ more clearly, in \cref{fig:amg-pcg-test1-iter}, we plot the number of iterations for two cases: $h = 2^{-7}$ and $h=2^{-9}$.

\begin{figure}[H]
	\centering
	\includegraphics[width=13cm]{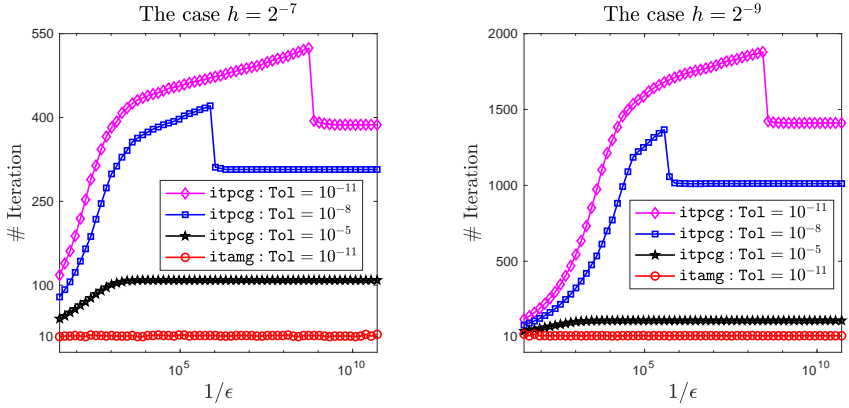}
	\caption{Growth behaviors of $\mathtt{itamg}$ and  $\mathtt{itpcg}$ with respect to the parameter $\epsilon$ and the tolerance $\mathtt{Tol}$. 
	}
	\label{fig:amg-pcg-test1-iter}
\end{figure}
\renewcommand\arraystretch{1.3}
\begin{table}[H]
	\centering
	\caption{The number of levels and the operator complexity of AMG.}
	\label{tab:amg-test1-nnz}
	\setlength{\tabcolsep}{8pt}
	\begin{tabular}{|c|c|c|c|c|c|c|c|c|c|c|}
		\hline
		\multirow{2}{*}{$1/h$} &\multicolumn{2}{c|}{$\epsilon = 10^{-4}$}
		&\multicolumn{2}{c|}{$\epsilon = 10^{-6}$}
		&\multicolumn{2}{c|}{$\epsilon = 10^{-8}$}
		&\multicolumn{2}{c|}{$\epsilon = 10^{-10}$}
		&\multicolumn{2}{c|}{$\epsilon = 0$}\\ \cline{2-11}
		
		&$J$ &$\mathtt{opcom}$ 
		&$J$ &$\mathtt{opcom}$ 
		&$J$  &$\mathtt{opcom}$ 
		&$J$ &$\mathtt{opcom}$ 
		&$J$ &$\mathtt{opcom}$ 	\\ \hline

		$2^{4} $              & 4  & 1.47  & 4  & 1.49  & 4 & 1.41 & 4 & 1.50 & 4 & 1.40 \\ \hline
		$2^{6} $              & 5  & 1.64  & 5  & 1.62  & 5 & 1.65 & 5 & 1.62 & 5 & 1.65 \\ \hline
		$2^{8} $              & 6  & 1.66  & 6  & 1.68  & 6 & 1.67 & 6 & 1.67 & 6 & 1.66 \\ \hline
		$2^{10} $            & 7  & 1.68  & 7  & 1.69  & 7 & 1.68 & 7 & 1.69 & 7 & 1.69 \\ \hline
		
	\end{tabular}
\end{table}

Except for the number of iterations, we also record two crucial ingredients of the multilevel hierarchy: (i) the number of levels $J$ and (ii) the {\it operator complexity} ($ \mathtt{opcom}$ for short)
\[
\mathtt{opcom} := \frac{\sum_{\ell=1}^{J}\mathtt{nnz}(A_\ell)}{ \mathtt{nnz}(A)}.
\]
The quantity $ \mathtt{opcom}$ is often used to measure the computational complexity of the AMG algorithm. From  \cref{tab:amg-test1-nnz}, we might observe the growth magnitude $\mathcal O(|\ln h|)$, as mentioned in \cref{rem:level-J}. This is almost negligible and thus both $J$ and $ \mathtt{opcom}$ are robust with respect to $h$ and $\epsilon$.

\subsection{The overall IPD-SsN-AMG method}
\label{sec:num-part2}
Combining \cref{algo:Inexact-PD,algo:SsN,algo:Hybrid_Solver,algo:MG-W}, we obtain the overall semismooth Newton-AMG-based Inexact Primal-Dual  (IPD-SsN-AMG for short) method for solving the generalized transport problem \cref{eq:GOT-x}; see \cref{algo:ipd-ssn-amg}.

\begin{algorithm}[H]
	\caption{The overall IPD-SsN-AMG method}
	\label{algo:ipd-ssn-amg}
	\begin{algorithmic}[1] 
		\REQUIRE  KKT tolerance: $\mathtt{KKT\_Tol}$.\\
		\qquad~SsN iteration tolerance: $\mathtt{SsN\_Tol}$.\\
		\qquad~Maximal SsN iteration number: $j_{\max}\in\mathbb N$.\\		
		\qquad~Line search parameters: $\tau \in(0,1/2),\,\delta \in(0,1)$.\\
		\qquad~Initial guesses: $\beta_0>0,\,\lambda_0\in\R^{m+n+r}$ and $\bm u_0=\bm v_0\in\R^{mn+n+m}$.\\
		\FOR{$k=0,1,\cdots$}
		\STATE Choose the step size $\alpha_k>0$.
		\STATE Set $	\tau_{k} ={} \beta_k(1+\alpha_k)/\alpha_k^2$ and $\eta_{k} =\sigma+\tau_{k}$.
		\STATE Set $D_{k} ={\rm diag}(\eta_{k}I_{mn},\tau_{k}I_n,\tau_{k}I_m)$ and $\bm w_k={} \widetilde{c}+\beta_{k}(\bm u_k+\alpha_k\bm v_k)/\alpha_k^2$.		
		\STATE Update $\displaystyle \beta_{k+1}= \beta_k/(1+\alpha_k)$ and set $	\widetilde{\lambda}_k={}\beta_{k+1}\left[\lambda_k-\beta_{k}^{-1}(H\bm u_k- b)\right]-b$.
		\STATE Set $ \lambda_{\rm new} = \lambda_k$.
		\FOR[SsN iteration]{$j = 0,1,\cdots$}
		\STATE Set $\lambda_{\rm old} = \lambda_{\rm new}$ and $\bm z_k= D^{-1}_{k}(\bm w_k-H^\top\lambda_{\rm new})$.
		\STATE Compute the diagonal matrix
		$U_k \in \partial \proj_{\Sigma}\left(\bm z_k\right)$ from \cref{eq:par-proj}.
		\STATE Transform the linear equation 
		\[
		\mathcal J_k\zeta= \left( \beta_{k+1}I+ HD_k^{-1}U_kH^{\top}\right)\zeta = -F_k(\lambda_{\rm old})
		\]
		into the reduced graph Laplacian system \cref{eq:equi-calA-eps}. 
		\STATE Apply \cref{algo:Hybrid_Solver,algo:MG-W} to \cref{eq:equi-calA-eps} and recover the solution $\zeta$.   \quad \{AMG\} \label{algo:amg}
		\STATE Update $\lambda_{\rm new} = \lambda_{\rm old} + \delta^\ell \zeta$ with the smallest nonnegative integer $\ell\in\mathbb N$ that satisfies $
		\mathcal F_k(\lambda_{\rm old}+\delta^\ell \zeta)\leq \mathcal F_k(\lambda_{\rm old})+\tau\delta^\ell\dual{F_k(\lambda_{\rm old}),\zeta}$. \quad \{Line search\}
		\IF[Check the SsN iteration]{$\nm{F_k(\lambda_{\rm new})}\leq\mathtt{SsN\_Tol}$ or $j\geq j_{\rm max}$}\label{algo:ssn}
		\STATE {\bf break}
		\ENDIF				
		\ENDFOR
		\STATE Update $\lambda_{k+1} = \lambda_{\rm new},\,\bm u_{k+1} = \proj_{\Sigma}\left(\bm z_k\right)$ and $\bm v_{k+1} = \bm u_{k+1}+(\bm u_{k+1}-\bm u_k)/\alpha_k$.
		\IF[Check the KKT residual]{${\rm Res}(k+1)\leqslant \mathtt{KKT\_Tol}$}\label{algo:kkt}
		\STATE {\bf break}
		\ENDIF		
		\ENDFOR
	\end{algorithmic}
\end{algorithm}

Detailed parameter choices and operations are explained in order. In step \ref{algo:amg}, the settings of the AMG $W$-cycle are the same as that in \cref{sec:num-part1-amg}. Note that  \cref{algo:Hybrid_Solver} requires the connected components of the graph $\mathcal G = (\mathcal V,\mathcal E)$ with respect to the Laplacian matrix $\mathcal A_0$ in \cref{eq:equi-calA-eps}. This can be done by using graph searching algorithms \cite{jungnickel_graphs_2005} such as breadth first search (with the complexity $\mathcal O(|\mathcal V||\mathcal E|)$) and depth first search (with the complexity $\mathcal O(|\mathcal E|) = \mathcal O(\mathtt{nnz}(\mathcal A_0))$). Thanks to the bipartite structure, we adopt the MATLAB built-in function $\mathtt{dmperm}$ that provides the Dulmage--Mendelsohn decomposition of $\mathcal A_0$ and also returns the connected components.

For SsN iteration, the line search parameters are $ \tau = 0.2$ and $\delta = 0.9$, and in step \ref{algo:ssn}, it shall be terminated when either $j$ is larger than the maximal iteration number $j_{\rm max} = 15$ or $\nm{F_k(\lambda_{\rm new})}$ is smaller than the tolerance $\mathtt{SsN\_Tol} = \max\{\beta_k(k+1)^{-2},\,10^{-11}\}$.

Moreover, in step \ref{algo:kkt} we  impose the stop criterion
\begin{equation}\label{eq:kkt }
	{\rm Res}(k) :=\max\left\{\frac{{\rm KKT}(x_k)}{{\rm KKT}(x_0)},\,\frac{{\rm KKT}(y_k)}{{\rm KKT}(y_0)},\,\frac{{\rm KKT}(z_k)}{{\rm KKT}(z_0)},\,\frac{{\rm KKT}(\lambda_k)}{{\rm KKT}(\lambda_0)}\right\}\leq \mathtt{KKT\_Tol},
\end{equation}
where $\mathtt{KKT\_Tol}>0$ denotes the tolerance and the KKT residuals are defined by
\[
\left\{
\begin{aligned}
	{\rm KKT}(x_k): ={}& \nm{x_k-\proj_{\mathcal X}(\sigma\phi+(1-\sigma)x_k-c-G^\top\lambda_k)},\\
	{\rm KKT}(y_k): = {}&\nm{y_k-\proj_{\mathcal Y}(y_k-I_Y^\top\lambda_k)},\\
	{\rm KKT}(z_k): = {}&\nm{z_k-\proj_{\mathcal Z}(z_k-I_Z^\top\lambda_k)},\\
	{\rm KKT}(\lambda_k): ={}&\nm{Gx_k+I_Yy_k+I_Zz_k-b}.	
\end{aligned}
\right.
\]

In the sequel, we investigate the performance of our IPD-SsN-AMG method on 
specific problems including optimal transport, Birkhoff projection and partial optimal transport. Also, comparisons with the semismooth Newton-based augmented Lagrangian methods proposed in \cite{li_asymptotically_2020,li_efficient_2020} and the accelerated ADMM method in \cite{luo_unified_2021} will be presented, under the same stopping condition \cref{eq:kkt } with $\mathtt{KKT\_Tol} = 10^{-6}$.

The methods in \cite{li_asymptotically_2020,li_efficient_2020} adopt PCG as the linear system solver, and the (super-)linear convergence analysis is based on classical proximal point framework together with proper error bound assumption. For convenience, we abbreviate these two methods simply as ALM-SsN-PCG. The method in \cite{luo_unified_2021}, denoted shortly by Acc-ADMM, possesses sublinear rates $\mathcal O(1/k)$ and $\mathcal O(1/k^2)$ respectively for convex and partially strongly convex objectives. 

We note that, as summarized in the introduction part, some other optimization solvers, such as entropy regularization methods and interior-point methods, can also be applied to transport-like problems considered here. However, entropy-based methods provide approximate solutions only with a fixed tolerance, which is almost the same magnitude as the regularization parameter. Interior-point methods utilize the barrier function and require linear system solver as well. Hence, it would be interesting to studying the efficiency comparison between PCG and AMG, and we leave this as our future topic.
\subsubsection{Optimal transport}
\label{sec:num-part2-ot}
Let us focus on the optimal mass transport \cref{eq:DOT-X} with $m=n\in\mathbb N$. The mass distributions $\mu,\,\nu\in \R^n_+$ are generated randomly, and we consider two kinds of cost matrices:
\begin{itemize}
	\item {\it Random cost}: 
	\begin{equation}\label{eq:rand-cost}
		C = \left(C_{ij}\right)_{n\times n}\quad\text{with}\quad C_{ij}\sim \mathcal U([0,1]),
	\end{equation}
	where $\mathcal U([0,1])$ denotes the uniform distribution on $[0,1]$;
	\item {\it Quadratic distance cost}: 
	\begin{equation}\label{eq:l2-cost}
		C = \left(C_{ij}\right)_{n\times n}\quad\text{with}\quad C_{ij}=\nm{x_i-x_j}^2,
	\end{equation}
	where $\{x_i\}_{i=1}^n$ are the grid points in the uniform subdivision of $\Omega = (0,1)^2$ with mesh size $1/h = \sqrt{n}-1$; see \cref{fig:amg-pcg-test1-mesh}.
\end{itemize}

As discussed in \cref{rem:ak-linear}, our IPD-SsN-AMG converges at least linearly as long as the step size is bounded below $\alpha_k\geq \alpha_0>0$. Practically, we are not allowed to increase $\alpha_k$ as large as we can. Hence, we choose $\alpha_k \geq 1$ for small $k(\leq 10)$ and set $\alpha_k \in(0,1)$ for large $k$. 
For ALM-SsN-PCG, there are two crucial parameters $\sigma_k$ and $\tau_k$; see equation (18) in \cite[Algorithm 1]{li_asymptotically_2020}. Theoretically, letting $\sigma_k$ increase to $\infty$ and $\tau_k$ decrease to $\tau_{\infty}>0$ implies superlinear convergence. However, for the sake of practical computation, we set the moderate choice: $\sigma_k =\mathcal O(k^2)$ and $\tau_k=\mathcal O(1/k)$. Additionally, we provide a warming-up initial guess for both two algorithms by running Acc-ADMM 100 times.

\renewcommand\arraystretch{1.2}
\begin{table}[H]
	\setlength{\tabcolsep}{4pt}
	\begin{tabular}{|c|c|c|cc|c|c|cc|c|c|}
		\hline
		&\multicolumn{4}{c|}{IPD-SsN-AMG} &\multicolumn{4}{c|}{ALM-SsN-PCG} &\multicolumn{2}{c|}{Acc-ADMM}\\ 
		\cline{2-11}
		\multirow{2}{*}{$m = n$} & 	\multirow{2}{*}{$\mathtt{itIPD}$} & \multirow{2}{*}{$\mathtt{itSsN}$} &\multicolumn{2}{c|}{$\mathtt{itamg}$ }  & \multirow{2}{*}{$\mathtt{itALM}$} & \multirow{2}{*}{$\mathtt{itssn}$} &\multicolumn{2}{c|}{$\mathtt{itpcg}$ }&\multirow{2}{*}{it}&\multirow{2}{*}{residual}\\ 		
		\cline{4-5}												\cline{8-9}		
		&   &   &max  &aver&  & &max&aver&&\\
		\hline
		$1000$ &19 &170 &13&7 &46 &286 &731&190  &5000&6.42e-02  \\ \hline
		$2000$ &29 &233 &14&7 &54 &416 &1299&235  &5000&1.05e-01  \\ \hline
		$3000$ &29 & 279&15&7 &57 &463 &2059&284  &5000&3.18e-01   \\ \hline
		$4000$ &39 &311 &13&6 &61 &531 &2100&264  &5000&4.95e-01   \\ \hline							
	\end{tabular}
	\caption{Numerical results for optimal transport with random cost \cref{eq:rand-cost}. } 
	\label{tab:class1-result-rand}
\end{table}
\begin{figure}[H]
	\centering
	\includegraphics[width=10cm]{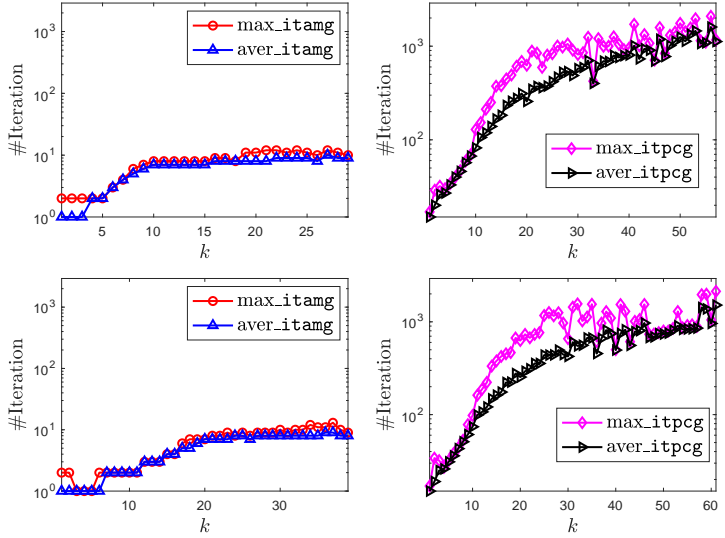}
	\caption{Growth behaviors of $\mathtt{itamg}$ and  $\mathtt{itpcg}$ for optimal transport with random cost: $m=n=3000$ for top row and $m=n=4000$ for bottom row.}
	\label{fig:class1-random}
\end{figure}

Numerical results with random cost  \cref{eq:rand-cost} and quadratic distance cost \cref{eq:l2-cost} are listed in Table \ref{tab:class1-result-rand} and Table \ref{tab:class1-result-l2}, respectively. We record (i) the number of iterations ($\mathtt{itIPD}$ and $\mathtt{itALM}$), (ii) the total number of SsN iterations ($\mathtt{itSsN}$ and $\mathtt{itssn}$), and (iii) the maximum (max) and average (aver) iteration number of AMG ($\mathtt{itamg}$) and PCG ($\mathtt{itpcg}$). Besides, Acc-ADMM is stopped at $k = 5000$ and we report the corresponding relative KKT residuals.

We find that $\mathtt{itIPD}$ ($\mathtt{itSsN}$) is better than $\mathtt{itALM}$ ($\mathtt{itssn}$) for random cost but slightly inferior for quadratic distance cost. Particularly, $\mathtt{itamg}$ is more robust than $\mathtt{itpcg}$, and we also plot the growth behaviors in \cref{fig:class1-random,fig:class1-l2}. As we can see, $\mathtt{itamg}$ stays around 10 while $\mathtt{itpcg}$ increases dramatically as $k$ does.

\renewcommand\arraystretch{1.2}
\begin{table}[H]
	\setlength{\tabcolsep}{4pt}
	
	\begin{tabular}{|c|c|c|cc|c|c|cc|c|c|}
		\hline
		&\multicolumn{4}{c|}{IPD-SsN-AMG} &\multicolumn{4}{c|}{ALM-SsN-PCG} &\multicolumn{2}{c|}{Acc-ADMM}\\ 
		\cline{2-11}
		\multirow{2}{*}{$m = n$} & 	\multirow{2}{*}{$\mathtt{itIPD}$} & \multirow{2}{*}{$\mathtt{itSsN}$} &\multicolumn{2}{c|}{$\mathtt{itamg}$ }  & \multirow{2}{*}{$\mathtt{itALM}$} & \multirow{2}{*}{$\mathtt{itssn}$} &\multicolumn{2}{c|}{$\mathtt{itpcg}$ }&\multirow{2}{*}{it}&\multirow{2}{*}{residual}\\ 		
		\cline{4-5}												\cline{8-9}		
		&   &   &max  &aver&  & &max&aver&&\\
		\hline
		$900$ &29 &215 &11&7 &20 &183 &628&163  &5000&1.03e-01  \\ \hline
		$1600$ &29 &225 &11&7 &21 &226 &979&191  &5000&1.69e-01 \\ \hline
		$2500$ &43 &328 &12&7 &25 &277 &1556&278  &5000&2.51e-01  \\ \hline
		$3600$ &40 &352 &13&7 &35 &360 &1798&640  &5000&3.48e-01  \\ \hline						
	\end{tabular}
	\caption{Numerical results for optimal transport with quadratic distance cost \cref{eq:l2-cost}.} 
	\label{tab:class1-result-l2}
\end{table}
\begin{figure}[H]
	\centering
	\includegraphics[width=10cm]{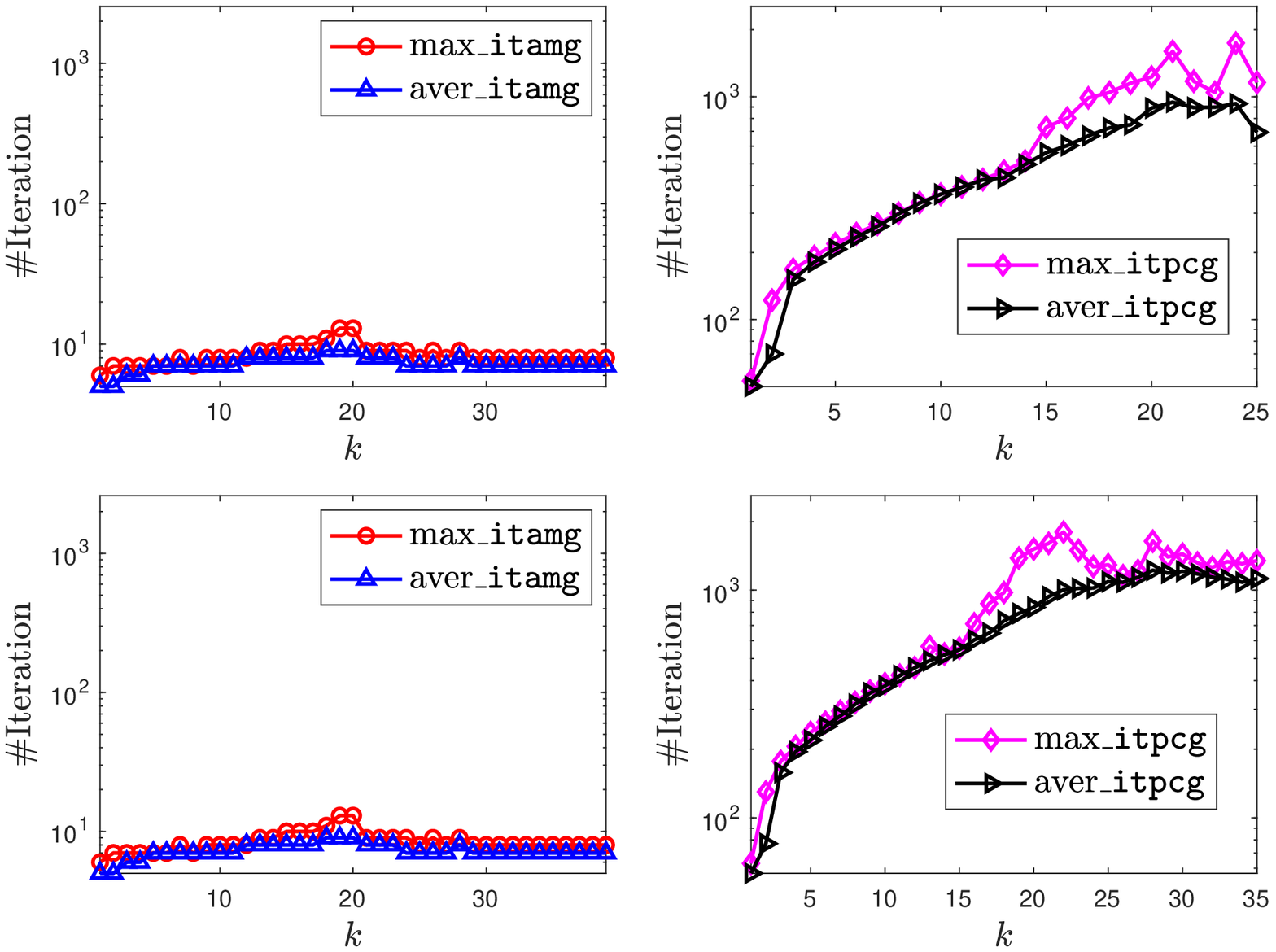}
	\caption{Growth behaviors of $\mathtt{itamg}$ and  $\mathtt{itpcg}$ for optimal transport with quadratic distance cost: $m=n=2500$ for top row and $m=n=3600$ for bottom row.}
	\label{fig:class1-l2}
\end{figure}
\renewcommand\arraystretch{1.2}
\begin{table}[H]
	\setlength{\tabcolsep}{4pt}
	
	\begin{tabular}{|c|c|c|cc|c|c|cc|c|c|}
		\hline
		&\multicolumn{4}{c|}{IPD-SsN-AMG} &\multicolumn{4}{c|}{ALM-SsN-PCG} &\multicolumn{2}{c|}{Acc-ADMM}\\ 
		\cline{2-11}
		\multirow{2}{*}{$m = n$} & 	\multirow{2}{*}{$\mathtt{itIPD}$} & \multirow{2}{*}{$\mathtt{itSsN}$} &\multicolumn{2}{c|}{$\mathtt{itamg}$ }  & \multirow{2}{*}{$\mathtt{itALM}$} & \multirow{2}{*}{$\mathtt{itssn}$} &\multicolumn{2}{c|}{$\mathtt{itpcg}$ }&\multirow{2}{*}{it}&\multirow{2}{*}{residual}\\ 		
		\cline{4-5}												\cline{8-9}		
		&   &   &max  &aver&  & &max&aver&&\\
		\hline
		$2000$&6 &18 &1&1 &8 &25 &15&10 &5000&3.33e-03 \\ \hline
		$3000$&6 &17 &1&1 &7 &24 &14&10 &5000&3.32e-03 \\ \hline
		$4000$&6 &19 &1&1 &7 &24 &20&10 &5000&3.23e-03 \\ \hline
		$5000$&6 &19 &1&1 &7 &24 &26&10 &5000&3.20e-03 \\ \hline
	\end{tabular}
	\caption{Numerical outputs for Birkhoff projection without entry constraint. 	} 
	\label{tab:class2-nolimit-result}
\end{table}
\subsubsection{Birkhoff projection}
\label{sec:num-part2-proj}
We then move to the Birkhoff projection \cref{eq:proj-Bn} with possible entry constraint \cref{eq:Xij}. For this problem, we choose fixed large step size $\alpha_k = 10$ for our IPD-SsN-AMG. Numerical outputs with random data are presented in \cref{tab:class2-nolimit-result,tab:class2-result}. Notice that both two algorithms work well, and $\mathtt{itamg}$ is still superior than $\mathtt{itpcg}$ (which is also quite robust). This might be due to the strongly convex property of the problem itself.

	%

\renewcommand\arraystretch{1.2}
\begin{table}[H]
	\setlength{\tabcolsep}{4pt}
	
	\begin{tabular}{|c|c|c|cc|c|c|cc|c|c|}
		\hline
		&\multicolumn{4}{c|}{IPD-SsN-AMG} &\multicolumn{4}{c|}{ALM-SsN-PCG} &\multicolumn{2}{c|}{Acc-ADMM}\\ 
		\cline{2-11}
		\multirow{2}{*}{$m = n$} & 	\multirow{2}{*}{$\mathtt{itIPD}$} & \multirow{2}{*}{$\mathtt{itSsN}$} &\multicolumn{2}{c|}{$\mathtt{itamg}$ }  & \multirow{2}{*}{$\mathtt{itALM}$} & \multirow{2}{*}{$\mathtt{itssn}$} &\multicolumn{2}{c|}{$\mathtt{itpcg}$ }&\multirow{2}{*}{it}&\multirow{2}{*}{residual}\\ 		
		\cline{4-5}												\cline{8-9}		
		& &   &max  &aver&  & &max&aver&&\\
		\hline
		$1000$&6 &13 &1&1 &7 &21 &15&11 &5000&3.81e-03  \\ \hline
		$2000$ &6 &20 &1&1 &8 &24 &15&10 &5000&3.42e-03 \\ \hline
		$3000$ &6 &18 &1&1 &6 &30 &14&11 &5000&3.30e-03  \\ \hline
		$4000$ &6 &17 &1&1 &5 &42 &20&12 &5000&3.22e-03  \\ \hline
	\end{tabular}
	\caption{Numerical outputs for Birkhoff projection with entry constraint. } 
	\label{tab:class2-result}
\end{table}

	%

\subsubsection{Partial optimal transport}
\label{sec:num-part2-pot}
Finally, let us consider the problem of partial optimal transport \cref{eq:PDOT} with random cost \cref{eq:rand-cost} and quadratic distance cost \cref{eq:l2-cost}. Again, the marginal distributions $\mu$ and $\nu$ and the fraction of mass $a$ are generated randomly. 

From \cref{tab:class4-result-random,tab:class4-result-l2}, we observe that: (i) similar with the results of optimal transport (see \cref{tab:class1-result-rand,tab:class1-result-l2}), $\mathtt{itIPD}$ is much less than $\mathtt{itALM}$ for random cost but slightly more than that for quadratic distance cost; (ii) $\mathtt{itSsN}$ is better than $\mathtt{itssn}$ for both two cases; (iii) $\mathtt{itamg}$ stays robust and outperforms $\mathtt{itpcg}$.

Growth behaviors of $\mathtt{itamg}$ and $\mathtt{itpcg}$ are displayed in \cref{fig:class4-random,fig:class4-l2}. We find that $\mathtt{itamg}$ is temperately increasing within few initial steps while $\mathtt{itpcg}$ is not robust with respect to both the iteration process (i.e., the number $k$) and the problem size.

\renewcommand\arraystretch{1.2}
\begin{table}[H]
	\setlength{\tabcolsep}{4pt}
	
	\begin{tabular}{|c|c|c|cc|c|c|cc|c|c|}
		\hline
		&\multicolumn{4}{c|}{IPD-SsN-AMG} &\multicolumn{4}{c|}{ALM-SsN-PCG} &\multicolumn{2}{c|}{Acc-ADMM}\\ 
		\cline{2-11}
		\multirow{2}{*}{$m = n$} & 	\multirow{2}{*}{$\mathtt{itIPD}$} & \multirow{2}{*}{$\mathtt{itSsN}$} &\multicolumn{2}{c|}{$\mathtt{itamg}$ }  & \multirow{2}{*}{$\mathtt{itALM}$} & \multirow{2}{*}{$\mathtt{itssn}$} &\multicolumn{2}{c|}{$\mathtt{itpcg}$ }&\multirow{2}{*}{it}&\multirow{2}{*}{residual}\\ 		
		\cline{4-5}												\cline{8-9}		
		&   &   &max  &aver&  & &max&aver&&\\
		\hline
		$1000$ &20 &152 &35&14  &66 &274 &452&154   &5000&1.55e-01 \\ \hline
		$2000$ &34 &205 &25&8 &72 &352 &644&145 &5000&4.12e-01  \\ \hline
		$3000$ &34 &225 &23&6 &74 &411 &426&87  &5000&8.71e-01\\ \hline
		$4000$ &33 &238 &29&6 &81 &462 &555&100 &5000&3.22e+01\\ \hline
	\end{tabular}
	\caption{Numerical results for partial optimal transport with random cost \cref{eq:rand-cost}. } 
	\label{tab:class4-result-random}
\end{table}
\renewcommand\arraystretch{1.2}
\begin{table}[H]
	\setlength{\tabcolsep}{4pt}
	
	\begin{tabular}{|c|c|c|cc|c|c|cc|c|c|}
		\hline
		&\multicolumn{4}{c|}{IPD-SsN-AMG} &\multicolumn{4}{c|}{ALM-SsN-PCG} &\multicolumn{2}{c|}{Acc-ADMM}\\ 
		\cline{2-11}
		\multirow{2}{*}{$m = n$} & 	\multirow{2}{*}{$\mathtt{itIPD}$} & \multirow{2}{*}{$\mathtt{itSsN}$} &\multicolumn{2}{c|}{$\mathtt{itamg}$ }  & \multirow{2}{*}{$\mathtt{itALM}$} & \multirow{2}{*}{$\mathtt{itssn}$} &\multicolumn{2}{c|}{$\mathtt{itpcg}$ }&\multirow{2}{*}{it}&\multirow{2}{*}{residual}\\ 		
		\cline{4-5}												\cline{8-9}		
		&   &   &max  &aver&  & &max&aver&&\\
		\hline
		$900$ &31 &154 &19&5  &18 &139 &95&60   &5000&5.95e-01  \\ \hline
		$1600$ &32 &155 &28&5 &22 &176 &98&63 &5000&1.77e-01  \\ \hline
		$2500$ &32 &196 &36&6 &25 &223 &128&64  &5000&6.20e-01 \\ \hline
		$3600$ &31 &204 &46&7 &29 &244 &138&66 &5000&5.60e-01\\ \hline
	\end{tabular}
	\caption{Numerical results for partial optimal transport with quadratic distance cost \cref{eq:l2-cost}. } 
	\label{tab:class4-result-l2}
\end{table}

	%

\begin{figure}[H]
	\centering
	\includegraphics[width=10cm]{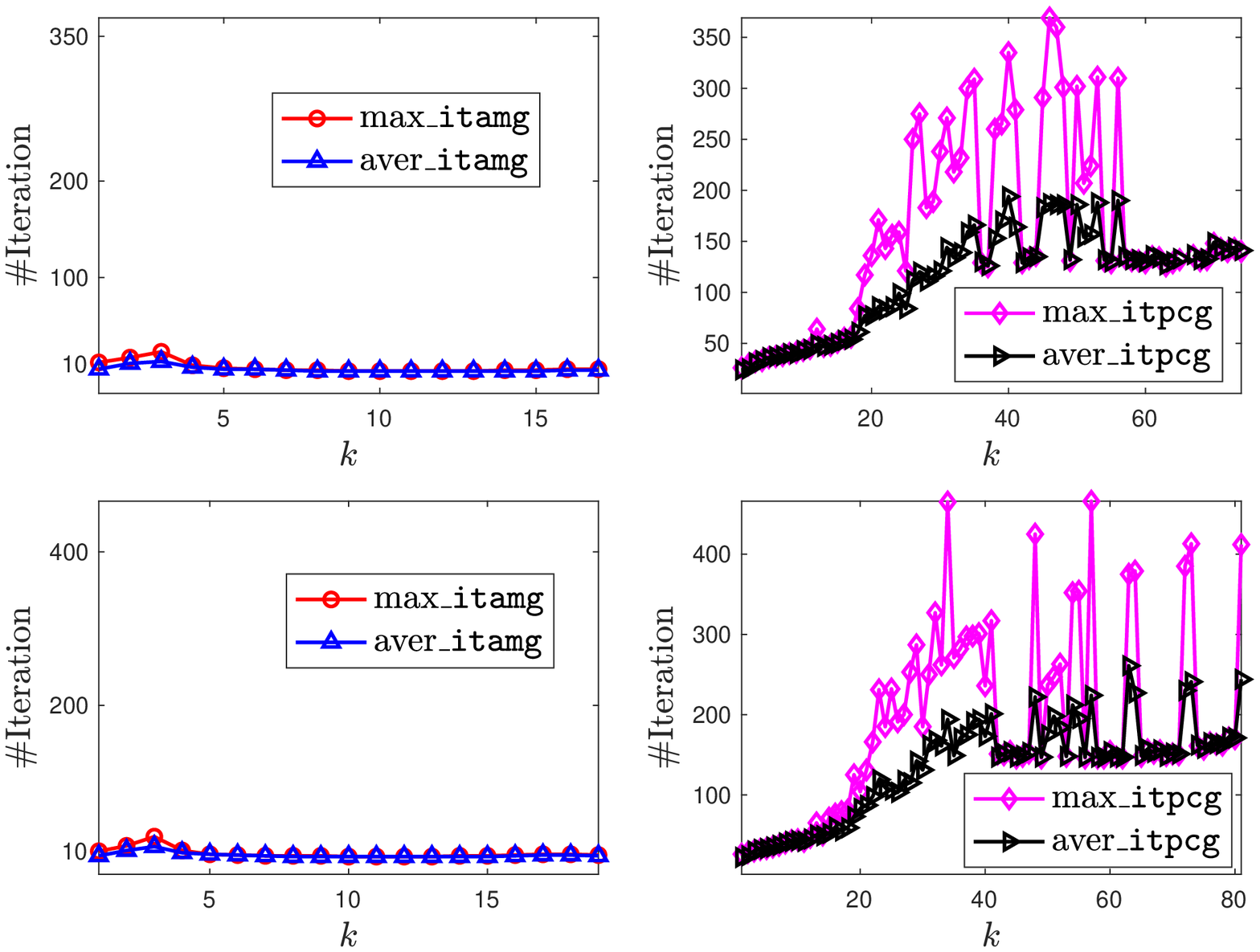}
	\caption{Growth behaviors of $\mathtt{itamg}$ and  $\mathtt{itpcg}$ for partial optimal transport with random cost: $m=n=3000$ for top row and $m=n=4000$ for bottom row.}
	\label{fig:class4-random}
\end{figure}

\begin{figure}[H]
	\centering
	\includegraphics[width=10cm]{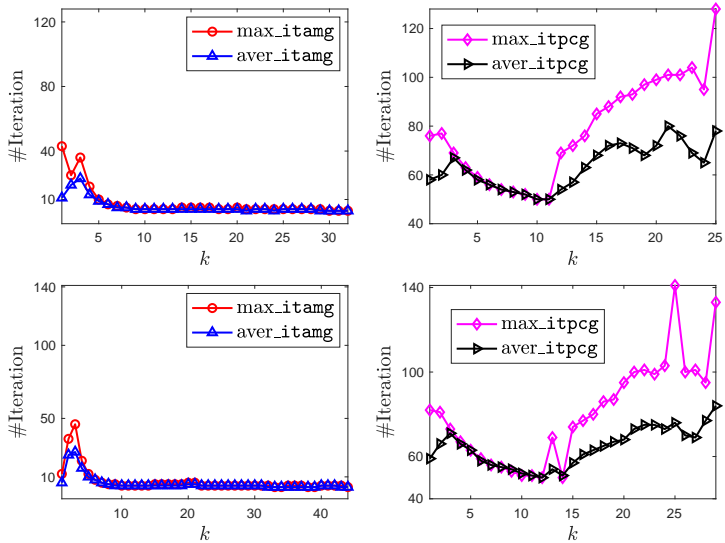}
	\caption{Growth behaviors of $\mathtt{itamg}$ and  $\mathtt{itpcg}$ for partial optimal transport with quadratic distance cost: $m=n=2500$ for top row and $m=n=3600$ for bottom row.}
	\label{fig:class4-l2}
\end{figure}

	%

%
\section{Conclusion}
\label{sec:con}
In this paper, we propose an efficient semismooth Newton-AMG-based inexact primal-dual method for a large class of transport-like problems that share the common feature of marginal distribution constraint. We follow the differential equation solver approach and prove the (super-)linear convergence rate via discrete Lyapunov function. Utilizing the hidden graph structure of the linear system arising from the semismooth Newton iteration, we use the algebraic multilevel method and establish a robust estimate of the two-level case by the Xu-Zikatanov identity. Extensive numerical experiments are also provided to validate the performance of our algorithm.

\bibliographystyle{abbrv}
\bibliography{mylibrary}

\begin{thebibliography}{10}

\bibitem{ahuja_network_1993}
B.~K. Ahuja.
\newblock {\em Network {Flows}: {Theory}, {Algorithms}, and {Applications}}.
\newblock Prentice Hall, 1993.

\bibitem{altschuler_near-linear_2017}
J.~Altschuler, J.~Niles-Weed, and P.~Rigollet.
\newblock Near-linear time approximation algorithms for optimal transport via
  {Sinkhorn} iteration.
\newblock In {\em 31st {Conference} on {Neural} {Information} {Processing}
  {Systems} ({NIPS} 2017)}, Long Beach, CA, USA, 2017.

\bibitem{arjovsky_wasserstein_2017}
M.~Arjovsky, S.~Chintala, and L.~Bottou.
\newblock Wasserstein generative adversarial networks.
\newblock In {\em Proceedings of the 34 th {International} {Conference} on
  {Machine} {Learning}}, volume~70, Sydney, Australia, 2017. PMLR.

\bibitem{bai_computing_2007}
Z.~Bai, D.~Chu, and R.~C.~E. Tan.
\newblock Computing the nearest doubly stochastic matrix with a prescribed
  entry.
\newblock {\em SIAM J. Sci. Comput.}, 29(2):635--655, 2007.

\bibitem{Bai_partial_2001}
Z.~Bai and R.~Freund.
\newblock A partial pade-via-{Lanczos} method for reduced-order modeling.
\newblock {\em Linear Algebra Appl.}, 332:139--164, 2001.

\bibitem{Bauschke2011}
H.~Bauschke and P.~Combettes.
\newblock {\em {Convex Analysis and Monotone Operator Theory in Hilbert
  Spaces}}.
\newblock CMS Books in Mathematics. Springer Science+Business Media, New York,
  2011.

\bibitem{benamou_optimal_2021}
J.-D. Benamou.
\newblock Optimal transportation, modelling and numerical simulation.
\newblock {\em Acta Numer.}, 30:249--325, 2021.

\bibitem{benamou_iterative_2015}
J.-D. Benamou, G.~Carlier, M.~Cuturi, L.~Nenna, and G.~Peyré.
\newblock Iterative {Bregman} projections for regularized transportation
  problems.
\newblock {\em SIAM J. Sci. Comput.}, 37(2):A1111--A1138, 2015.

\bibitem{Bertsekas_auction_1992}
D.~P. Bertsekas.
\newblock Auction algorithms for network flow problems: a tutorial
  introduction.
\newblock {\em Comput. Optim. Appl.}, 1(1):7--66, 1992.

\bibitem{Bertsimas1997}
D.~Bertsimas and J.~Tsitsiklis.
\newblock {\em Introduction to Linear Optimization}.
\newblock Athena Scientific, 1997.

\bibitem{hutchison_algebraic_2006}
R.~Blaheta.
\newblock Algebraic {Multilevel} {Methods} with {Aggregations}: {An}
  {Overview}.
\newblock In {\em Large-{Scale} {Scientific} {Computing}}, volume 3743, pages
  3--14. Springer Berlin Heidelberg, Berlin, 2006.

\bibitem{bramble_parallel_1990}
J.~H. Bramble, J.~E. Pasciak, and J.~Xu.
\newblock Parallel multilevel preconditioners.
\newblock {\em Math. Comput.}, 55(191):1--22, 1990.

\bibitem{brauer_sinkhorn-newton_2018}
C.~Brauer, C.~Clason, D.~Lorenz, and B.~Wirth.
\newblock A {Sinkhorn}-{Newton} method for entropic optimal transport.
\newblock {\em arXiv:1710.06635}, 2018.

\bibitem{brenner_mathematical_2008}
S.~C. Brenner and L.~R. Scott.
\newblock {\em The {Mathematical} {Theory} of {Finite} {Element} {Methods}}.
\newblock Number~15 in Texts in {Applied} {Mathematics}. Springer, New York,
  NY, 3rd edition, 2008.

\bibitem{brezis_remarks_2018}
H.~Brezis.
\newblock Remarks on the {Monge}--{Kantorovich} problem in the discrete
  setting.
\newblock {\em Comptes Rendus Mathematique}, 356(2):207--213, 2018.

\bibitem{briggs_multigrid_2000}
W.~L. Briggs, V.~E. Henson, and S.~F. McCormick.
\newblock {\em A {Multigrid} {Tutorial}}.
\newblock Society for Industrial and Applied Mathematics, USA, 2nd edition,
  2000.

\bibitem{Brualdi2006}
R.~Brualdi.
\newblock {\em Combinatorial {Matrix} {Classes}}.
\newblock Cambridge University Press, New York, 2006.

\bibitem{burkard_assignment_2009}
R.~E. Burkard, M.~Dell'Amico, and S.~Martello.
\newblock {\em Assignment {Problems}}.
\newblock SIAM, Society for Industrial and Applied Mathematics, Philadelphia,
  2009.

\bibitem{caffarelli_free_2010}
L.~Caffarelli and R.~McCann.
\newblock Free boundaries in optimal transport and {Monge}-{Ampère} obstacle
  problems.
\newblock {\em Ann. Math.}, 171(2):673--730, 2010.

\bibitem{chapel_partial_2020}
L.~Chapel and M.~Z. Alaya.
\newblock Partial optimal transport with applications on positive-unlabeled
  learning.
\newblock In {\em 34th {Conference} on {Neural} {Information} {Processing}
  {Systems} ({NeurIPS} 2020)}, Vancouver, Canada, 2020.

\bibitem{Chen:2008ifem}
L.~Chen.
\newblock {$i$FEM}: an integrated finite element methods package in {MATLAB}.
\newblock Technical report, 2009.

\bibitem{huang_deriving_2011}
L.~Chen.
\newblock Deriving the {X}-{Z} identity from auxiliary space method.
\newblock In Y.~Huang, R.~Kornhuber, O.~Widlund, and J.~Xu, editors, {\em
  Domain {Decomposition} {Methods} in {Science} and {Engineering} {XIX}},
  volume~78, pages 309--316. Springer, Berlin, 2011.

\bibitem{chen_convergence_2020}
L.~Chen, X.~Hu, and S.~M. Wise.
\newblock Convergence analysis of the fast subspace descent methods for convex
  optimization problems.
\newblock {\em Math. Comput.}, 89(325):2249--2282, 2020.

\bibitem{chen_optimal_2012}
L.~Chen, R.~H. Nochetto, and J.~Xu.
\newblock Optimal multilevel methods for graded bisection grids.
\newblock {\em Numer. Math.}, 120:1--34, 2012.

\bibitem{chernov_fast_2016}
A.~Chernov, P.~Dvurechensky, and A.~Gasnikov.
\newblock Fast primal-dual gradient method for strongly convex minimization
  problems with linear constraints.
\newblock In Y.~Kochetov, M.~Khachay, V.~Beresnev, E.~Nurminski, and
  P.~Pardalos, editors, {\em 9th {Discrete} {Optimization} and {Operations}
  {Research}}, volume 9869 of {\em Lecture {Notes} in {Computer} {Science}},
  pages 391--403, Vladivostok, Russia, 2016. Springer, Cham.

\bibitem{chizat_scaling_2018}
L.~Chizat, G.~Peyré, B.~Schmitzer, and F.-X. Vialard.
\newblock Scaling algorithms for unbalanced optimal transport problems.
\newblock {\em Math. Comput.}, 87(314):2563--2609, 2018.

\bibitem{clarke_optimization_1987}
F.~H. Clarke.
\newblock {\em Optimization and {Nonsmooth} {Analysis}}.
\newblock Number~5 in Classics in {Applied} {Mathematics}. Society for
  Industrial and Applied Mathematics, 1987.

\bibitem{cominetti_asymptotic_1994}
R.~Cominetti and J.~S. Martín.
\newblock Asymptotic analysis of the exponential penalty trajectory in linear
  programming.
\newblock {\em Math. Program.}, 67(1-3):169--187, 1994.

\bibitem{courty2017optimal}
N.~Courty, R.~Flamary, D.~Tuia, and A.~Rakotomamonjy.
\newblock Optimal transport for domain adaptation.
\newblock {\em IEEE Transactions on Pattern Analysis and Machine Intelligence},
  39(9):1853--1865, 2017.

\bibitem{cuturi_sinkhorn_2013}
M.~Cuturi.
\newblock Sinkhorn distances: {L}ightspeed computation of optimal transport.
\newblock In {\em Advances in Neural Information Processing Systems 26}, pages
  2292--2300, 2013.

\bibitem{dennis_numerical_1996}
J.~E. Dennis and R.~B. Schnabel.
\newblock {\em Numerical {Methods} for {Unconstrained} {Optimization} and
  {Nonlinear} {Equations}}.
\newblock Number~16 in Classics in applied mathematics. Society for Industrial
  and Applied Mathematics, Philadelphia, 1996.

\bibitem{dvurechensky_computational_2018}
P.~Dvurechensky, A.~Gasnikov, and A.~Kroshnin.
\newblock Computational optimal transport: complexity by accelerated gradient
  descent is better than by {Sinkhorn}'s algorithm.
\newblock In {\em Proceedings of the 35 th {International} {Conference} on
  {Machine} {Learning}}, volume~80, Stockholm, Sweden, 2018. PMLR.

\bibitem{eckstein_alternating_1990}
J.~Eckstein and D.~Bertsekas.
\newblock An alternating direction method for linear programming.
\newblock Technical report LIDS-P-1967, Cambridge, 1990.

\bibitem{Facchinei2003-v2}
F.~Facchinei and J.~Pang.
\newblock {\em {Finite-Dimensional Variational Inequalities and Complementarity
  Problems, vol 2}}.
\newblock Springer, New York, 2003.

\bibitem{Ferguson_1956}
A.~R. Ferguson and G.~B. Dantzig.
\newblock The allocation of aircraft to routes--{A}n example of linear
  programming under uncertain demand.
\newblock {\em Management Science}, 3(1), 1956.

\bibitem{figalli_optimal_2010}
A.~Figalli.
\newblock The optimal partial transport problem.
\newblock {\em Arch. Ration. Mech. Anal.}, 195(2):533--560, 2010.

\bibitem{fogel_convex_2015}
F.~Fogel, R.~Jenatton, F.~Bach, and A.~d'Aspremont.
\newblock Convex relaxations for permutation problems.
\newblock In {\em Advances in {Neural} {Information} {Processing} {Systems}
  26}, pages 1016--1024, 2013.

\bibitem{gasnikov_efficient_2016}
A.~V. Gasnikov, E.~B. Gasnikova, Y.~E. Nesterov, and A.~V. Chernov.
\newblock Efficient numerical methods for entropy-linear programming problems.
\newblock {\em Comput. Math. Math. Phys.}, 56(4):514--524, 2016.

\bibitem{glunt_nearest_1998}
W.~Glunt, T.~L. Hayden, and R.~Reams.
\newblock The nearest ``doubly stochastic'' matrix to a real matrix with the
  same first moment.
\newblock {\em Numer. Linear Algebr. Appl.}, 5(6):475--482, 1998.

\bibitem{guminov_accelerated_2021}
S.~Guminov, P.~Dvurechensky, N.~Tupitsa, and A.~Gasnikov.
\newblock Accelerated alternating minimization, accelerated {Sinkhorn}’s
  algorithm and accelerated iterative {Bregman} projections.
\newblock {\em arXiv:1906.03622}, 2021.

\bibitem{hackbusch_multi-grid_2011}
W.~Hackbusch.
\newblock {\em Multi-{Grid} {Methods} and {Applications}}.
\newblock Springer, Berlin, 2011.

\bibitem{hug_multi-physics_2015}
R.~Hug, E.~Maitre, and N.~Papadakis.
\newblock Multi-physics optimal transportation and image interpolation.
\newblock {\em ESAIM: Math. Model. Numer. Anal.}, 49(6):1671--1692, 2015.

\bibitem{jiang_l_p-norm_2016}
B.~Jiang, Y.-F. Liu, and Z.~Wen.
\newblock $l_p$-norm regularization algorithms for optimization over
  permutation matrices.
\newblock {\em SIAM J. Optim.}, 26(4):2284--2313, 2016.

\bibitem{jungnickel_graphs_2005}
D.~Jungnickel.
\newblock {\em Graphs, {Networks}, and {Algorithms}}.
\newblock Number~5 in Algorithms and Computation in Mathematics. Springer,
  Berlin, 2nd edition, 2005.

\bibitem{kandasamy2018neural}
K.~Kandasamy, W.~Neiswanger, J.~Schneider, B.~Poczos, and E.~Xing.
\newblock Neural architecture search with {B}ayesian optimisation and optimal
  transport.
\newblock In {\em Advances in Neural Information Processing Systems 31}, 2018.

\bibitem{kantorovich_translocation_1942}
L.~Kantorovich.
\newblock On the translocation of masses.
\newblock {\em Dokl. Akad. Nauk. USSR (N.S.)}, 37:199--201, 1942.

\bibitem{khoury_closest_1998}
R.~N. Khoury.
\newblock Closest matrices in the space of generalized doubly stochastic
  matrices.
\newblock {\em J. Math. Anal. Appl.}, 222(2):562--568, 1998.

\bibitem{korman_optimal_2014}
J.~Korman and R.~J. McCann.
\newblock Optimal transportation with capacity constraints.
\newblock {\em Trans. Am. Math. Soc.}, 367(3):1501--1521, 2014.

\bibitem{lee_robust_2007}
Y.-J. Lee, J.~Wu, J.~Xu, and L.~Zikatanov.
\newblock Robust subspace correction methods for nearly singular systems.
\newblock {\em Math. Models Meth. Appl. Sci.}, 17(11):1937--1963, 2007.

\bibitem{lee2014path}
Y.~T. Lee and A.~Sidford.
\newblock Path finding methods for linear programming: Solving linear programs
  in $\tilde{O}(\sqrt{rank})$ iterations and faster algorithms for maximum
  flow.
\newblock In {\em 2014 IEEE 55th Annual Symposium on Foundations of Computer
  Science}, pages 424--433. IEEE, 2014.

\bibitem{li_two-level_2015}
B.~Li and X.~Xie.
\newblock A two-level algorithm for the weak {Galerkin} discretization of
  diffusion problems.
\newblock {\em J. Comput. Appl. Math.}, 287:179--195, 2015.

\bibitem{li_bpx_2016}
B.~Li and X.~Xie.
\newblock {BPX} preconditioner for nonstandard finite element methods for
  diffusion problems.
\newblock {\em SIAM J. Numer. Anal.}, 54(2):1147--1168, 2016.

\bibitem{li_analysis_2016}
B.~Li, X.~Xie, and S.~Zhang.
\newblock Analysis of a two-level algorithm for {HDG} methods for diffusion
  problems.
\newblock {\em Commun. Comput. Phys.}, 19(5):1435--1460, 2016.

\bibitem{Li2019}
H.~Li and Z.~Lin.
\newblock Accelerated alternating direction method of multipliers: {A}n optimal
  ${O}(1/{K})$ nonergodic analysis.
\newblock {\em J. Sci. Comput.}, 79(2):671--699, 2019.

\bibitem{li_asymptotically_2020}
X.~Li, D.~Sun, and K.-C. Toh.
\newblock An asymptotically superlinearly convergent semismooth {N}ewton
  augmented {L}agrangian method for {L}inear {P}rogramming.
\newblock {\em SIAM J. Optim.}, 30(3):2410--2440, 2020.

\bibitem{li_efficient_2020}
X.~Li, D.~Sun, and K.-C. Toh.
\newblock On the efficient computation of a generalized {Jacobian} of the
  projector over the {Birkhoff} polytope.
\newblock {\em Math. Program.}, 179(1-2):419--446, 2020.

\bibitem{liao_fast_2022}
Q.~Liao, J.~Chen, Z.~Wang, B.~Bai, S.~Jin, and H.~Wu.
\newblock Fast {Sinkhorn} {I}: {An} ${O}({N})$ algorithm for the
  {Wasserstein}-1 metric.
\newblock {\em arXiv:2202.10042}, 2022.

\bibitem{Lim_beyond_2014}
C.~H. Lim and S.~J. Wright.
\newblock Beyond the {Birkhoff} polytope: convex relaxations for vector
  permutation problems.
\newblock In {\em Advances in Neural Information Processing Systems}, pages
  2168--2176, 2014.

\bibitem{lin_augmented_2022}
M.~Lin, D.~Sun, and K.-C. Toh.
\newblock An augmented {Lagrangian} method with constraint generations for
  shape-constrained convex regression problems.
\newblock {\em Math. Program.}, 14:223--270, 2022.

\bibitem{lin_efficient_2019}
T.~Lin, N.~Ho, and M.~I. Jordan.
\newblock On efficient optimal transport: {An} analysis of greedy and
  accelerated mirror descent algorithms.
\newblock In {\em International {Conference} on {Machine} {Learning}}, pages
  3982--3991. PMLR, 2019.

\bibitem{liu_multiscale_2022}
Y.~Liu, Z.~Wen, and W.~Yin.
\newblock A multiscale semi-smooth {Newton} method for optimal transport.
\newblock {\em J. Sci. Comput.}, 91(2):1--39, 2022.

\bibitem{luo_accelerated_2021}
H.~Luo.
\newblock Accelerated differential inclusion for convex optimization.
\newblock {\em Optimization}, https://doi.org/10.1080/02331934.2021.2002327,
  2021.

\bibitem{luo_accelerated_pd_2021}
H.~Luo.
\newblock Accelerated primal-dual methods for linearly constrained convex
  optimization problems.
\newblock {\em arXiv:2109.12604}, 2021.

\bibitem{luo_unified_2021}
H.~Luo.
\newblock A unified differential equation solver approach for separable convex
  optimization: splitting, acceleration and nonergodic rate.
\newblock {\em arXiv:2109.13467}, 2021.

\bibitem{luo_primal-dual_2022}
H.~Luo.
\newblock A primal-dual flow for affine constrained convex optimization.
\newblock {\em ESAIM: Control, Optimisation and Calculus of Variations},
  28:10.1051/cocv/2022032, 2022.

\bibitem{levy_partial_2022}
B.~Lévy.
\newblock Partial optimal transport for a constant-volume {Lagrangian} mesh
  with free boundaries.
\newblock {\em J. Comput. Phys.}, 451:1--26, 2022.

\bibitem{maas_generalized_2015}
J.~Maas, M.~Rumpf, C.~Schönlieb, and S.~Simon.
\newblock A generalized model for optimal transport of images including
  dissipation and density modulation.
\newblock {\em ESAIM: Math. Model. Numer. Anal.}, 49(6):1745--1769, 2015.

\bibitem{mangasarian_lipschitz_1987}
O.~L. Mangasarian and T.-H. Shiau.
\newblock Lipschitz continuity of solutions of linear inequalities, programs
  and complementarity problems.
\newblock {\em SIAM J. Control Optim.}, 25(3):583--595, 1987.

\bibitem{merigot_optimal_2021}
Q.~Mérigot and B.~Thibert.
\newblock Optimal {Transport}: {Discretization} and {Algorithms}.
\newblock In {\em Handbook of {Numerical} {Analysis}}, volume~22, pages
  133--212. Elsevier, 2021.

\bibitem{padiy_generalized_2001}
A.~Padiy, O.~Axelsson, and B.~Polman.
\newblock Generalized augmented matrix preconditioning approach and its
  application to iterative solution of ill-conditioned algebraic systems.
\newblock {\em SIAM J. Matrix Anal. Appl.}, 22(3):793--818, 2001.

\bibitem{2016Amplitude}
V.~M. Panaretos and Y.~Zemel.
\newblock Amplitude and phase variation of point processes.
\newblock {\em Ann. Stat.}, 44(2):771--812, 2016.

\bibitem{pele2009fast}
O.~Pele and M.~Werman.
\newblock Fast and robust earth mover’s distances.
\newblock In {\em In 2009 {IEEE} 12th {International} {Conference} on
  {Computer} {Vision}}, pages 460--467, 2009.

\bibitem{Peyre2019}
G.~Peyr{\'{e}} and M.~Cuturi.
\newblock Computational optimal transport.
\newblock {\em Found. Trends Mach. Learn.}, 11(5-6):1--257, 2019.

\bibitem{Qi1993}
L.~Qi.
\newblock Convergence analysis of some algorithms for solving nonsmooth
  equations.
\newblock {\em Math. Oper. Res.}, 18(1):227--244, 1993.

\bibitem{Qi1993a}
L.~Qi and J.~Sun.
\newblock A nonsmooth version of {N}ewton's method.
\newblock {\em Math. Program.}, 58(1-3):353--367, 1993.

\bibitem{Robinson1973}
S.~Robinson.
\newblock Bounds for error in the solution set of a perturbed linear program.
\newblock {\em Linear Algebra Appl.}, 6(C):69--81, 1973.

\bibitem{2000The}
Y.~Rubner, C.~Tomasi, and L.~J. Guibas.
\newblock The earth mover’s distance as a metric for image retrieval.
\newblock {\em International Journal of Computer Vision}, 40(2):99--121, 2000.

\bibitem{saad_iterative_2003}
Y.~Saad.
\newblock {\em Iterative {Methods} for {Sparse} {Linear} {Systems}}.
\newblock Society for Industrial and Applied Mathematics, USA, 2nd edition,
  2003.

\bibitem{sinkhorn_diagonalequivalenceto_1967}
R.~Sinkhorn.
\newblock Diagonale quivalence to matrices with prescribed row and columnsums.
\newblock {\em The American Mathematical Monthly}, 74(4):402--405, 1967.

\bibitem{Szkely2004TESTINGFE}
G.~Sz{\'e}kely and M.~Rizzo.
\newblock Testing for equal distributions in high dimension.
\newblock In {\em Inter-Stat (London)}, pages 1--16, 2004.

\bibitem{trottenberg_multigrid_2001}
U.~Trottenberg, C.~W. Oosterlee, and A.~Schüller.
\newblock {\em Multigrid}.
\newblock Academic Press, San Diego, 2001.

\bibitem{villani_topics_2003}
C.~Villani.
\newblock {\em Topics in {O}ptimal {T}ransportation}.
\newblock American Mathematical Society, 2003.

\bibitem{villani_optimal_2009}
C.~Villani.
\newblock {\em Optimal {Transport}: {Old} and {New}}.
\newblock Number 338 in Grundlehren der mathematischen {Wissenschaften}.
  Springer, Berlin, 2009.

\bibitem{wu_convergence_2012}
Y.~Wu, L.~Chen, X.~Xie, and J.~Xu.
\newblock Convergence analysis of {V}-{Cycle} multigrid methods for anisotropic
  elliptic equations.
\newblock {\em IMA J. Numer. Anal.}, 32(4):1329--1347, 2012.

\bibitem{xu_theory_1989}
J.~Xu.
\newblock {\em Theory of {Multilevel} {Methods}}.
\newblock {PhD} {Thesis}, Cornell University, Ithaca, New York, 1989.

\bibitem{xu_new_1992}
J.~Xu.
\newblock A new class of iterative methods for nonself-adjoint or indefinite
  problems.
\newblock {\em SIAM J. Numer. Anal.}, 29(2):303--319, 1992.

\bibitem{xu_two-grid_1996}
J.~Xu.
\newblock Two-grid discretization techniques for linear and nonlinear {PDEs}.
\newblock {\em SIAM J. Numer. Anal.}, 33(5):1759--1777, 1996.

\bibitem{xu_multilevel_2017}
J.~Xu.
\newblock {\em Multilevel {Iterative} {Methods}}.
\newblock Lecture Notes. Penn State University, 2017.

\bibitem{xu_method_2002}
J.~Xu and L.~Zikatanov.
\newblock The method of alternating projections and the method of subspace
  corrections in {Hilbert} space.
\newblock {\em J. Am. Math. Soc.}, 15(3):573--597, 2002.

\bibitem{xu_algebraic_2017}
J.~Xu and L.~Zikatanov.
\newblock Algebraic multigrid methods.
\newblock {\em Acta Numer.}, 26:591--721, 2017.

\bibitem{Xu2017}
Y.~Xu.
\newblock Accelerated first-order primal-dual proximal methods for linearly
  constrained composite convex programming.
\newblock {\em SIAM J. Optim.}, 27(3):1459--1484, 2017.

\bibitem{zikatanov_two-sided_2008}
L.~Zikatanov.
\newblock Two-sided bounds on the convergence rate of two-level methods.
\newblock {\em Numer. Linear Algebr. Appl.}, 15(5):439--454, 2008.

\end{thebibliography}

\end{document}